\documentclass[11pt]{amsart}

\usepackage{hyperref}

\usepackage{tikz-network,comment}
\usepackage{lscape}

\usepackage[tmargin=1in,bmargin=1in,rmargin=0.8in,lmargin=0.8in]{geometry}

\usepackage{amssymb}
\usepackage{amsmath}
\usepackage{amsthm}
\usepackage{graphicx}
\usepackage{enumerate}
\usepackage{tikz}
\usepackage{tabularx}
\usetikzlibrary{arrows.meta}
\tikzset{>={Latex[width=1.2mm,length=1.7mm]}}

\theoremstyle{plain}
\newtheorem{thm}{Theorem}[section]

\newtheorem{prop}[thm]{Proposition}

\numberwithin{equation}{section}

\newtheorem{theorem}[equation]{Theorem}

\newtheorem{utheorem}{\textrm{\textbf{Theorem}}}

\theoremstyle{definition}
\newtheorem{example}[thm]{Example}
\newtheorem{defn}[thm]{Definition}
\newtheorem{remark}[thm]{Remark}

\newcommand{\R}{\mathbb{R}}
\newcommand{\sn}{\mathrm{S}_n}
\newcommand{\sm}{\mathrm{S}_m}
\newcommand{\csn}{\mathbb{C}[\sn]}

\newcommand{\K}{\mathcal{B}}
\newcommand{\imm}[1]{\mathrm{Imm}_{#1}}
\newcommand{\sumsb}[1]{\sum_{\substack{#1}}} 
\newcommand{\defeq}{:=}
\newcommand{\spn}{\mathrm{span}}
\newcommand{\sgn}{\mathrm{sgn}}
\newcommand{\ntnsp}{\negthinspace}
\newcommand{\permmon}[2]{#1_{1,#2_1} \ntnsp\cdots {#1}_{n,#2_n}}
\newcommand{\tn}{T_n(\xi)}
\newcommand{\Gr}{\mathrm{Gr}}

\begin{document}

\title[Pl\"ucker inequalities for weakly separated coordinates]{Pl\"ucker inequalities for weakly separated coordinates \\ in totally nonnegative Grassmannian}

\author{Daniel Soskin and Prateek Kumar Vishwakarma}

\address[D.~Soskin]{Department of Mathematics, Lehigh University, Chandler-Ullmann Hall, 17 Memorial Drive East, Bethlehem, PA 18015, USA}
\email{das523@lehigh.edu}

\address[P.K.~Vishwakarma]{Department of Mathematics and Statistics$,$ College West 307$.$14, 3737 Wascana Parkway$,$ University of Regina$,$ SK S$4$S $0$A$2,$ Canada}
\email{prateek.vishwakarma@uregina.ca,~prateekv@alum.iisc.ac.in}

\date{\today}



\bibliographystyle{dart}

\begin{abstract}
We show that the partial sums of the long Pl\"ucker relations for pairs of weakly separated Pl\"ucker coordinates oscillate around $0$ on the totally nonnegative part of the Grassmannian. Our result generalizes the classical oscillating inequalities by Gantmacher--Krein (1941) and recent results on totally nonnegative matrix inequalities by Fallat--Vishwakarma (2023). In fact we obtain a characterization of weak separability, by showing that no other pair of Pl\"ucker coordinates satisfies this property. 

Weakly separated sets were initially introduced by Leclerc and Zelevinsky and are closely connected with the cluster algebra of the Grassmannian. Moreover, our work connects several fundamental objects such as weak separability, Temperley--Lieb immanants, and Pl\"ucker relations, and provides a very general and natural class of additive determinantal inequalities on the totally nonnegative part of the Grassmannian.
\end{abstract}

\maketitle

\section{Introduction and main results}
A real matrix is called \textit{totally positive/nonnegative} (TP/TNN) if the determinants of all its square submatrices are positive (nonnegative). Total positivity arose initially in a few different areas. It was studied by Gantmacher--Krein \cite{GantKreinOsc} in oscillations of vibrating systems, by Fekete--P\'olya \cite{FP1912} (following Laguerre) in understanding the variation diminishing property of linear operators, by Schoenberg \cite{ASW} in applications to the analysis of real roots of polynomials and spline functions. These matrices play important roles in algebraic and enumerative combinatorics, integrable systems, probability, classical mechanics, and many other areas \cite{ando,FJTNNMatrices,gas2013total,K2}. The area of total positivity continues to be studied by many researchers, see for instance \cite{Brosowsky-Chepuri-Mason-JCTA, Chepuri-Sherman-Bennett-CJM, Khare-Chou-Kannan, Galashin-Karp-Lam, Galashin-Pylyavskyy, Parisi-Sherman-Bennett-Williams, Serhiyenko-Sherman-Bennett-Williams}. Lusztig extended the notion of total positivity to reductive Lie groups \emph{G} \cite{LusztigTP}, where the totally nonnegative part $\emph{G}^{\geq0}$ of \emph{G} is a semialgebraic subset of $\emph{G}$ generated by Chevalley generators. $\emph{G}^{\geq0}$ is the subset of $\emph{G}$ where all elements of the dual canonical basis are nonnegative \cite{lusztig1998introduction}. This concept could be generalized even further to varieties \emph{V}. The totally nonnegative subvariety is defined as a subset of \emph{V} where certain regular functions on \emph{V} have nonnegative values \cite{berenstein1997total,fomin1999db,gekhtman2010cluster}. Lusztig proved that specializations of elements of the dual canonical basis in the representation theory of quantum groups at $q=1$ are totally nonnegative polynomials. Thus, it is important to investigate classes of functions on matrices that are nonnegative on totally nonnegative matrices. We will discuss a source of such functions that are closely related to the \emph{(long) Pl\"ucker relations}.

This project has its origins in the classical work by Gantmacher--Krein \cite{GantKreinOsc} in which they modify the Laplace formula for the determinant, to obtain a sequence of inequalities oscillating about 0 and holding for all totally nonnegative matrices. These were extended to new systems of inequalities by Fallat--Vishwakarma \cite{fallat2023inequalities}, in which they similarly extracted an oscillating set of inequalities from the generalized Laplace identity. To state this result more precisely we introduce some notations: for a matrix $A,$ define $A_{P,Q}$ as the submatrix with rows in $P$ and columns in $Q,$ $[m,n]:=\{m,\dots,n\}$ for integers $m\leq n,$ and $[n]:=[1,n]$ if $1\leq n.$ 

\begin{theorem}[Fallat--Vishwakarma \cite{fallat2023inequalities}]\label{FVthm}
Let $1\leq d < n$ be integers. Suppose $P_d:=[1,d],$ and $Q_{dk} := [n-d,n]\setminus\{n-d+k\},$ for all $k\in [0,d].$ Then
\begin{align*}
\sum_{k=0}^{l} (-1)^{l+k} \det A_{P_{d},Q_{dk}} \det A_{[n]\setminus P_{d},[n]\setminus Q_{dk}} \geq 0, \mbox{ for all }l\in [0,d],
\end{align*}
for all $n\times n$ totally nonnegative matrices $A.$
\end{theorem}

We shall see in Section~\ref{Sec:GenLapIneq} that these generalized Laplace inequalities are a partial refinement of certain long Pl\"ucker relations using Temperley--Lieb immanants discussed by Rhoades--Skandera \cite{RSkanTLImmp}. Another classical identity that can also be seen as one of the Pl\"ucker relations is the famous Karlin's identity \cite{K2}, and Fallat--Vishwakarma completely refined this identity for TNN matrices (see Theorem~B in \cite{fallat2023inequalities}). These two instances suggest exploring the possible existence of a larger set of ``{Pl\"ucker inequalities}'' over the TNN (totally nonnegative) Grassmannian. This is precisely the content of our main theorem; to state the complete result, we need some classical facts.

Let $1\leq m\leq n$ be integers. Grassmannian $\Gr(m,m+n)$ is the manifold of $m$-dimensional subspaces $V$ of $\R^{m+n}.$ Such subspaces can be represented by matrices $X\in \R^{(m+n) \times m}$ of full rank. Let $X_I$ denote the submatrix corresponding to rows indexed by $m$ element subsets $I \subseteq [m+n]$ and columns indexed by $\{1,2,\dots,m\}.$ Here $\Delta_I(X):=\det X_{I}$ denotes the corresponding maximal minor of $X.$ Note that right multiplication by an invertible $m\times m$ matrix $B$ rescales each $\Delta_I(X)$ by $\det B.$ Thus, the minors $\Delta_I(X)$ together form projective coordinates of $\Gr(m,m+n),$ known as the corresponding \textit{Pl\"ucker coordinates.} These coordinates are related via \textit{Pl\"ucker relations} \cite{gekhtman2010cluster,postnikov2006total}:
\begin{align*}
\Delta_{(i_{1},\dots,i_{m})}\cdot \Delta_{(j_{1},\dots,j_{m})}= \sum_{k=1}^{m} \Delta_{(j_{k},i_{2},\dots,i_{m})}\cdot\Delta_{(j_{1},\dots,j_{k-1},i_{1},j_{k+1},\dots,j_{m})} 
\end{align*}
for all $i_1,\dots,i_{m},j_1,\dots,j_{m}\in [m+n].$ Here a function $\Delta_{(i_{1},...,i_{m})}$ is labelled by an ordered sequence, and equals $\Delta_{\{i_{1},...,i_{m}\}}$ if $i_{1} < \dots < i_{m}.$ Also, $\Delta_{(i_{1},...,i_{m})}=(-1)^{\sgn(w)}\Delta_{(i_{w(1)},...,i_{w(m)})}$ for all permutations $w \in \sm.$ The totally nonnegative (TNN) Grassmannian $\Gr^{\geq 0}(m,m+n)\subset \Gr(m,m+n)$ is the collection of those vector subspaces which have a representative matrix $X\in \R^{(m+n)\times m}$ with all $\Delta_{I} \geq 0.$ (To be precise, TNN Grassmannian refers to those subspaces that have a representative matrix $X\in \R^{(m+n)\times m}$ with all nonzero $\Delta_{I}$ with the same sign.) Similarly, the totally positive (TP) Grassmannian $\Gr^{> 0}(m,m+n)\subset \Gr^{\geq 0}(m,m+n)$ is the collection of those vector subspaces which have a representative matrix $X\in \R^{(m+n)\times m}$ with all $\Delta_{I} > 0.$ (For precision, TP Grassmannian refers to those subspaces that have a representative matrix $X\in \R^{(m+n)\times m}$ with all $\Delta_{I}$ nonzero and of the same sign.) See, for instance, \cite{postnikov2006total} for more details. In this paper we are concerned with certain inequalities in Pl\"ucker coordinates over the entire TNN Grassmannian. For $c_{I,J}\in \R$ and coordinates $\Delta_I ~\&~\Delta_J,$ we call 
\begin{align*}
\sum_{I,J} c_{I,J}\Delta_{I}\Delta_{J} \geq 0 \quad \mbox{over}\quad \Gr^{\geq 0}(m,m+n)
\end{align*}
if for all $V\in \Gr^{\geq 0}(m,m+n)$ and all representative matrices $X \in \R^{(m+n)\times m}$ of $V,$ $\sum_{I,J} c_{I,J}\Delta_{I}(X)\Delta_{J}(X) \geq 0.$ The main result in this paper unifies the classical \cite{GantKreinOsc} and recent \cite{fallat2023inequalities} oscillating inequalities mentioned above for TNN matrices, by refining the Pl\"ucker relations for $\Gr^{\geq 0}(m,m+n).$ To state it, we require some notations:

\begin{defn}\label{DIJlr}
Fix integers $1\leq m\leq n,$ and let $I$ and $J$ be ordered $m$ element subsets of $[m+n].$ Imagine that the elements of $I$ and $J$ are on the circle with points $1,2,\dots,m+n$ marked in clockwise order. 
\begin{enumerate}
\item We call $I,J$ \textit{weakly separated} if the sets $I\setminus J$ and $J\setminus I$ can be separated by a chord in the circle.
\item Suppose $\eta$ denotes the number of elements in sets $I\setminus J$ and $J\setminus I.$ Denote $I\setminus J := \{i_1,\cdots,i_{\eta}\}$ and $J \setminus I := \{j_1,\cdots,j_{\eta}\},$ such that $i_1~{<}_{c}~\dots~{<}_{c}~i_{\eta}~{<}_{c}~ j_{\eta}$ and $i_1~{<}_{c}~j_1~{<}_{c}~\dots~{<}_{c}~j_{\eta}$ in the clockwise order ${<}_{c},$ where each order is decided by moving in clockwise direction on the circle starting from $i_1.$ (For example: if $m=n=6,$ $I=(1,5,3,4,10,11)$ and $J=(2,6,7,8,9,11),$ then $(i_1,i_2,i_3,i_4,i_5)=(10,1,3,4,5)$ and $(j_1,j_2,j_3,j_4,j_5)=(2,6,7,8,9)$. In particular, we choose $i_1$ and $j_{\eta}$ such that all other $i_k,j_k$ lie in between them while we traverse in the clockwise order starting from $i_1.$) Now define the following pair of tuples for $k,r\in [1,\eta]$: 
\begin{equation}\label{Eq:IJlr}
I_{k,r}:=(\dots,j_k,\dots), \quad \mbox{and} \quad 
J_{k,r}:=(\dots,i_r,\dots),
\end{equation}
where $j_k$ and $i_r$ replace each other in $I$ and $J,$ respectively.
\end{enumerate}
\end{defn}

We begin by first describing our main result for weakly separated $I,J.$ In the following inequalities, denote by $I^{\uparrow}$ the re-ordering of $m$ element ordered subset $I$ of $[m+n]$ into increasing coordinates. Let $l,r\in [1,\eta],$ then:
\begin{equation}\label{Main-ex:1}
(-1)^{l}\Big{(}\sum^{l}_{k=1}(-1)^{k} \Delta_{I_{k,r}^{\uparrow}}\Delta_{J_{k,r}^{\uparrow}} + (-1)^{\eta-r}\Delta_{I^{\uparrow}}\Delta_{J^{\uparrow}}\Big{)} \geq 0 \quad \forall~ l \geq \eta-r+1,~ \mbox{ over }\Gr^{\geq 0}(m,m+n).
\end{equation}

This provides a novel class of ``Pl\"ucker-type'' determinantal inequalities that hold over the entire TNN Grassmannian for the class of weakly separated $I,J.$ It also suggests two natural follow-up questions:
\begin{enumerate}
\item Are there analogous Pl\"ucker-type inequalities that hold over $\Gr^{\geq 0}(m,m+n)$ for the remaining cases of $l,r \in [1,\eta]?$ (Here we continue to work with weakly separated $I,J.$)  
\item If yes, then does this also have a converse, in that certain Pl\"ucker-type determinantal inequalities do not hold if $I,J$ are not weakly separated?
\end{enumerate}

Our main theorem provides an affirmative answer to both of these questions, and in the process also characterizes weak separability. We introduce a notation before stating this result: for ordered $m$ element subsets $I:=(i_1,\dots,i_m)$ of $[m+n],$ $\sgn(I) = (-1)^{\sgn(w)}$ where $w\in \mathrm{S}_{m}$ such that $i_{w(1)} < \dots < i_{w(m)}.$
\begin{utheorem}\label{th:ws1}
Let $I,J$ be ordered $m$ element subsets of $[m+n];$ notation as in Definition~\ref{DIJlr}. Consider the following system of oscillating inequalities for $l,r\in [1,\eta]$ over $\Gr^{\geq 0}(m,m+n)$:
\begin{align}\label{ws11:eq3}
\begin{aligned}
\sgn(I_{l,r})\,\sgn(J_{l,r})\sum^{l}_{k=1} \Delta_{I_{k,r}}\Delta_{J_{k,r}}  & \geq 0 \quad \forall~ l < \eta-r+1, \mbox{ and} \\
\sgn(I_{l,r})\,\sgn(J_{l,r})\Big{(}\sum^{l}_{k=1}\Delta_{I_{k,r}}\Delta_{J_{k,r}} - \Delta_{I}\Delta_{J} \Big{)}  & \geq 0 \quad \forall~ l \geq \eta-r+1.  
\end{aligned}
\end{align}
This system holds for all $l,r\in [1,\eta]$ if and only if $I$ and $J$ are weakly separated.
\end{utheorem}

\begin{remark}[``Uniform'' reformulation]
Theorem~\ref{th:ws1} can be stated using more ``homogeneous'' terminology. Define for $k,r\in [1,\eta],$
\begin{equation}\label{ws11:eq1}
\Pi_{k,r} :=
\begin{cases}
\Delta_{I_{k,r}}\Delta_{J_{k,r}} - \Delta_{I}\Delta_{J}, & \mbox{if }k=\eta-r+1, \mbox{ and} \\
\Delta_{I_{k,r}}\Delta_{J_{k,r}}, &\mbox{otherwise.}
\end{cases}
\end{equation}
{\it Then the following system of oscillating inequalities holds:
\begin{equation}\label{ws11:eq2}
\sgn(I_{l,r})\,\sgn(J_{l,r})\sum_{k=1}^{l}\Pi_{k,r}\geq 0 \qquad \forall l,r\in [1,\eta],~\mbox{ over }\Gr^{\geq 0}(m,m+n),  
\end{equation}
if and only if $I$ and $J$ are weakly separated.}
\end{remark}

\begin{example}
Recent refinements of the classical Karlin's identity for TNN matrices (see Theorems~B in \cite{fallat2023inequalities}) can be seen as Pl\"ucker-type inequalities in Theorem~\ref{th:ws1} above. To state the inequality, let $m\geq 1$ be an integer. Suppose $T$ is a subset of $[m]$ and $V:=\{v_1<\cdots<v_{m'} \}:=[m]\setminus T.$ Then for all $p\in [m]$ and set $S,$ such that $S\subseteq [m] \setminus \{p\}$ and $|S|=|T|+1,$
\begin{equation}\label{Kar-GK-FV}
(-1)^{1+l}\sum_{k=1}^{l} (-1)^{1+k} \det A\big{(}S\big{|}T\cup \{v_k\}\big{)} \det A\big{(}[n]\setminus\{p\}\big{|}[n]\setminus\{v_k\}\big{)}\geq 0,
\end{equation}
for all $l\in [m'],$ and all $m\times m$ totally nonnegative matrices $A.$

Now, suppose $s:= \max S$ and let $\Gr^{\geq 0}(m,2m+1)_s \subset \Gr^{\geq 0}(m,2m+1)$ corresponds to those representative matrices in which rows $s$ and $s+1$ are identical. It is easy to see that $\Gr^{\geq 0}(m,2m+1)_s$ can be constructed from $\Gr^{\geq 0}(m,2m).$

Suppose $S=\{s_1<\dots < s_{m-m'} < s \}$ and $p<s.$ ($p>s$ can be dealt with similarly.) Now choose 
\[
I=(1,\dots,p-1,p+1,\dots,s,s+1,\dots,m+1),\quad  J = (s_1,\dots,s_{m-m'},2m+2-v_{m'},\dots,2m+2-v_1).
\]
Note that $I,J$ are weakly separated, and choosing $i_r=s$ provides us with the required system of inequalities over $\Gr^{\geq 0}(m,2m+1)_s$ which is equivalent to \eqref{Kar-GK-FV}.


\end{example}

Another special case of Theorem~\ref{th:ws1} consists of the Gantmacher--Krein inequalities \cite{GantKreinOsc}, as well as their generalizations in \cite{fallat2023inequalities} (see Theorem~\ref{FVthm} above). We discuss these inequalities in Section~\ref{Sec:GenLapIneq}. Before this, we recall in Section~\ref{Sec:Pre} the tools needed to prove Theorem~\ref{th:ws1}, including Temperley--Lieb immanants and Kauffman diagrams. In Section~\ref{forward-imp}, we prove one implication of Theorem~\ref{th:ws1}: the system \eqref{ws11:eq3} holds for weakly separated Pl\"ucker coordinates. We conclude by showing the reverse implication in Section~\ref{rev-imp}.

\section{Prerequisites}\label{Sec:Pre}

\subsection{TNN Grassmannian and TNN matrices}\label{TNNGrass}

Let $1\leq m \leq n$ be integers. The set of $n\times m$ TNN matrices can be embedded inside in $\Gr^{\geq 0}(m,m+n)$ via:
\begin{align}\label{Plucker-embed}
\{\mbox{all $n\times m$ TNN matrices}\}\hookrightarrow \Gr^{\geq 0}(m,m+n)\quad\mbox{where}\quad A \mapsto \overline{A}= \begin{pmatrix} A \\ W_{0} \end{pmatrix}
\end{align}
\begin{align*}
\hspace*{-1cm}\mbox{where }W_{0}:=(w_{ij})=\big((-1)^{i+1}\cdot\delta_{j,m-i+1}\big)^{m}_{i,j=1}, \mbox{ i.e. } w_{ij}=\begin{cases} (-1)^{i+1} & \text{if } j=m-i+1,\\ 0 & \text{otherwise.}\end{cases}.
\end{align*}
\noindent It is not very difficult to see that there is a one-to-one correspondence between the minors of $A$ and the maximal minors of $\overline{A}.$ For precision, let $P\subseteq [n],Q \subseteq [m]$ with $|P|=|Q|,$ and note that:
\begin{align*}
\det A_{P,Q}=\det \overline{A}_{I,[m]}=:\Delta_{I}(\overline{A}), \mbox{ where } I:= P \cup \{ m+n+1-j \,|\, j \in [m] \setminus Q\}.
\end{align*}
Conversely, for $n$ element $I\subseteq [m+n],$
\begin{equation*}
\Delta_{I}(\overline A)=\det A_{I \cap [n],[m]\setminus \{m+n+1-i | i \in I \cap [n+1,...,m+n]\}}.
\end{equation*}
This one-to-one correspondence provides us with the following equivalence between inequalities that are quadratic in minors of TNN matrices and inequalities that are quadratic in Pl\"ucker coordinates over the TNN Grassmannian. Here and onward, we interchangeably use $\Delta_{P,Q}(A)$ and $\det A_{P,Q}$ to denote the determinant of the submatrix $A_{P,Q},$ and we follow the convention $\Delta_{\emptyset,\emptyset}(A)=1.$

\begin{theorem}\label{Ineq-equivalence}
Let $1\leq m\leq n$ be integers, and suppose sets $P_i\subseteq [n]$ and $Q_i\subseteq [m]$ with $|P_i|=|Q_i|,$ for $i=1,2.$ Let $c_{P_1,Q_1,P_2,Q_2}=c_{I,J}\in \R,$ where each 
\begin{align*}
I=P_1\cup \{ m+n+1-j \,|\, j \in [m] \setminus Q_1\}\quad \mbox{and}\quad J=P_2\cup \{ m+n+1-j \,|\, j \in [m] \setminus Q_2\}.
\end{align*}
Then the following inequalities are equivalent: 
\begin{align*}
\sum_{P_1,Q_1,P_2,Q_2} c_{P_1,Q_1,P_2,Q_2}\det A_{P_1,Q_1} \det A_{P_2,Q_2} & \geq 0\quad \forall A_{n\times m} \quad TNN. \\
\sum_{I,J} c_{I,J}\Delta_{I}(\overline{A}) \Delta_{J}(\overline{A}) &\geq 0 \quad \forall A_{n\times m} \quad TNN. \\
\sum_{I,J} c_{I,J}\Delta_{I} \Delta_{J} &\geq 0\quad \mbox{ over } \quad \Gr^{\geq 0}(m,m+n).
\end{align*}
\end{theorem}
\begin{proof}
The equivalence of the first two inequalities follows from \eqref{Plucker-embed}. The equivalence between the second and third follows from the density of the \textit{positive} Grassmannian $\Gr^{>}(m,m+n)$ in $\Gr^{\geq 0}(m,m+n),$ and that the Pl\"ucker coordinates form projective coordinates. For completeness, we provide the details. Suppose $A\in \R^{(m+n)\times m}$ be a representative matrix corresponding to $V \in \Gr^{>}(m,m+n).$ There exists $m\times m$ real matrix $B$ such that $A=\begin{pmatrix} A' \\ W_{0} \end{pmatrix} B,$ where $W_{0}:=(w_{ij})=\big((-1)^{i+1}\cdot\delta_{j,m-i+1}\big)^{m}_{i,j=1}.$ Since $\det W_{0} = 1>0,$ all other maximal minors of $\begin{pmatrix} A' \\ W_{0} \end{pmatrix}$ are positive as well, since the Pl\"ucker coordinates are projective coordinates. In particular, $A'$ is totally nonnegative. This implies that the second inequality holds for $A'.$ Using these arguments, we can say that
$$
\sum_{I,J} c_{I,J}\Delta_{I}(A) \Delta_{J}(A)=\sum_{I,J} c_{I,J} (\det B)^2 \Delta_{I}(\overline{A'}) \Delta_{J}(\overline{A'}) = (\det B)^2 \sum_{I,J} c_{I,J}  \Delta_{I}(\overline{A'}) \Delta_{J}(\overline{A'}) \geq 0.    
$$
Therefore, the third inequality holds over the TP Grassmannian, given that the second inequality holds. To prove that it holds over the TNN Grassmannian, we use density. To prove the density, consider the Gaussian kernel matrix $G_{\sigma,m+n} := (e^{-\sigma(j-k)^2})_{j,k=1}^{m+n},$ and suppose $A$ is a representative matrix corresponding to $V\in \Gr^{\geq 0}(m,m+n).$ It is well known that $G_{\sigma,m+n}$ is a totally positive matrix (see \cite{AK-book} for instance). Using Cauchy--Binet formula, one can see that $G_{\sigma,m+n} A$ corresponds to some $U\in \Gr^{>}(m,m+n).$ Density follows since $G_{\sigma,m+n} A \to A$ as $\sigma \to \infty.$ Continuity implies that the third inequality holds over the entire TNN Grassmannian.
\end{proof}

In Subsections~\ref{imm&nonnegativity-1} and \ref{imm&nonnegativity-2}, we shall discuss the nonnegativity of polynomials $p(x)$ with matrix entries, and thus we shall use $X:=(x_{ij})$ instead of $A:=(a_{ij}).$ Also, we use $p(X)$ (instead of $p(x)$) if the matrix $X$ is understandable from the context (for instance in Theorems~\ref{tm:plprod} and \ref{tm:plcomb}).

\subsection{Immanants and nonnegativity}\label{imm&nonnegativity-1}
Some determinantal inequalities can be stated in terms of polynomials in the matrix entries. A polynomial in the matrix entries is called \emph{totally nonnegative} if it attains nonnegative values on totally nonnegative matrices.  Following Littlewood \cite{LittlewoodTGC} and Stanley \cite{StanPos}, we call a particular type of such polynomials {\em immanants}. Given a function $f: \sn \rightarrow \mathbb C$ define the \emph{$f$-immanant} to be the polynomial
\begin{equation}\label{eq:immdef}
\imm f(x) \defeq \sum_{w \in \sn} f(w) \permmon xw \in \mathbb C[x].
\end{equation}
A particular family of immanants are defined for the monoid $\K_n$ generated by the standard basis of a \emph{Temperley--Lieb algebra}
$t_1,\dotsc,t_{n-1}$ subject to the relations
\begin{alignat*}{2}
t_i^2 &= \xi t_i, &\qquad &\text{for } i=1,\dotsc,n-1, \\
t_i t_j t_i &= t_i,   &\qquad &\text{if }  |i-j|=1,\\
t_i t_j &= t_j t_i,   &\qquad &\text{if }  |i-j| \geq 2.
\end{alignat*}
When $\xi = 2$ we have the isomorphism $T_n(2) \cong \csn/(1 + s_1 + s_2 + s_1s_2 + s_2s_1 + s_1s_2s_1).$ (See for instance \cite{FanMon}, \cite[Sections\,2.1 and \,2.11]{GHJ}, and \cite[Section\,7]{WestburyTL}.)
Specifically, the isomorphism is given by
\begin{equation}\label{eq:sntotn}
\begin{aligned}
\sigma : \csn \rightarrow \tn ~\mbox{ with }~ s_i \mapsto t_i - 1.
\end{aligned}
\end{equation}
It is known that $|\K_n|$ is the $n^{\mbox{th}}$ Catalan number $C_n = \tfrac{1}{n+1}\tbinom{2n}{n}.$ Diagrams of the basis elements of $T_n(\xi),$ made popular by Kauffman~\cite[Section\,4]{KauffState} are (undirected) graphs with $2n$ vertices and $n$ noncrossing edges, such that each edge lies in the convex hull spanned by $2n$ vertices (similarly to one of the definitions of Catalan numbers involving $n$ noncrossing chords on the circle with $2n$ vertices). We define a Temperley--Lieb immanant 
$\imm{K}(x)$ for each $K \in \K_n$ in terms of the function
\begin{equation}\label{eq:ftau}
\begin{aligned}
f_K: \csn \rightarrow \mathbb{R} ~\mbox{ with }~ w \mapsto \text{ coefficient of $K$ in } \sigma(w),
\end{aligned}
\end{equation}
and extend it linearly. To economize notation, we will write $\imm{K}$  instead of $\imm{f_K}$,
\begin{equation*}
\imm{K}(x) 
= \sum_{w \in \sn} f_K(w)x_{1,w_1} \cdots x_{n,w_n}.
\end{equation*}
It was shown by Rhoades--Skandera \cite{RSkanTLImmp} that \emph{Temperley--Lieb immanants} are a basis of the space 
\begin{equation}\label{eq:tlspace}
\spn_{\mathbb R} \{ \det X_{P,Q}  \det X_{P^c,Q^c}  \,|\, P,Q \subseteq [n] \mbox{ with }|P|=|Q| \}
\end{equation} 
and that they are TNN. Furthermore, these are the extreme rays of the cone of TNN immanants in this space~\cite[Theorem\,10.3]{RSkanTLImmp}.
\begin{thm}[Rhoades--Skandera \cite{RSkanTLImmp}]\label{SB}
Each immanant of the form
\begin{equation}\label{eq:sumofprodsof2minors}
\imm f(x) = \sumsb{P,Q \subseteq [n]\\ |P|=|Q|} c_{P,Q}
\det X_{P,Q} \det X_{P^c, Q^c} 
\end{equation}
is a totally nonnegative polynomial if and only if it is equal to a nonnegative linear combination of Temperley--Lieb immanants.
\end{thm}
In fact each complementary product of minors is a $0$-$1$ linear
combination of Temperley--Lieb immanants~\cite[Proposition\,4.4]{RSkanTLImmp}. We now restate this theorem for products of complementary Pl\"ucker coordinates of the representative $\overline{X}$ of the $n\times n$ totally nonnegative matrix $X.$
\begin{thm}[Rhoades--Skandera \cite{RSkanTLImmp}]\label{t:oneprod}
For $I \subseteq [2n]$, $|I|=n$ we have 
\begin{equation}\label{eq:pl1}
\Delta_{I}(\overline{X})\Delta_{I^c}(\overline{X}) =\sum_{K \in \K_n}b_{K}\imm{K}(x),    
\end{equation}
$
\mbox{where, }~  b_{K}=\begin{cases}
      1 & \text{if each edge in $K$ connects an element from $I$ and $I^c$}, \\
      0&\text{otherwise.}
    \end{cases}
$
\end{thm}
 
One of the possible methods to show the nonnegativity of an element from \eqref{eq:tlspace} is to count 2-colorings of the Kauffman diagrams, where each edge connects vertices of opposite colors. This has been discussed in detail by Skandera--Soskin \cite{skan2022sosk}, and we show this via the following example. In it and in what follows, we adopt the following \textbf{notation}: given a subset $I = \{ i_1, \dots, i_n \}$ of $[2n]$ with $i_1 < \cdots < i_n$, we write $\Delta_{[i_1, \dots, i_n]}:= \Delta_{(i_1, \dots, i_n)} = \Delta_{\{ i_1, \dots, i_n \}}$.

\begin{example}\label{ex:1} 
For any $3 \times 3$ totally nonnegative matrix $X$ the following inequality holds.
\begin{align*}
\Delta_{[123]}(\overline{X})\Delta_{[456]}(\overline{X}) \leq \Delta_{[124]}(\overline{X})\Delta_{[356]}(\overline{X}).
\end{align*}
\end{example}
\noindent Indeed, following Theorem~\ref{t:oneprod} we do get the inequality because
$$    
\Delta_{[123]}(\overline{X})\Delta_{[456]}(\overline{X})=\imm{K_1}(x), \mbox{ and }\Delta_{[124]}(\overline{X})\Delta_{[356]}(\overline{X})=\imm{K_1}(x)+\imm{K_2}(x),   
$$
where the diagrams $K_1$ and $K_2$ are shown in Figure~\ref{eq:ex1}. Note that vertices corresponding to the sets $I$ and $I^c$ are colored in white and black respectively. The first coloring admits only one bi-colored matching $K_1,$ while the second coloring admits exactly two bi-colored matchings $K_1$ and $K_2$. Thus, 
\begin{align*}
\Delta_{[124]}(\overline{X})\Delta_{[356]}(\overline{X}) - \Delta_{[123]}(\overline{X})\Delta_{[456]}(\overline{X}) =  \big{(}\imm{K_1}(x)+\imm{K_2}(x)\big{)} - \imm{K_1}(x)
 = \imm{K_2}(x)
 \geq 0.
\end{align*}

\begin{figure}
\begin{center}
\begin{tikzpicture}
\Vertex[x=0,y=0,size=0.1,color=black,position=right,fontscale=1.5,label=$v_6$]{6}
\Vertex[x=0,y=1,size=0.1,color=black,position=right,fontscale=1.5,label=$v_5$]{5}
\Vertex[x=0,y=2,size=0.1,color=black,position=right,fontscale=1.5,label=$v_4$]{4}

\Vertex[x=-1,y=2,size=0.1,color=white,position=left,fontscale=1.5,label=$v_3$]{3}
\Vertex[x=-1,y=1,size=0.1,color=white,position=left,fontscale=1.5,label=$v_2$]{2}
\Vertex[x=-1,y=0,size=0.1,color=white,position=left,fontscale=1.5,label=$v_1$]{1}

\Edge[](1)(6) \Edge(2)(5) \Edge(3)(4)

\Vertex[x=4,y=0,size=0.1,color=black,position=right,fontscale=1.5,label=$v_6$]{66}
\Vertex[x=4,y=1,size=0.1,color=black,position=right,fontscale=1.5,label=$v_5$]{55}
\Vertex[x=4,y=2,size=0.1,color=white,position=right,fontscale=1.5,label=$v_4$]{44}

\Vertex[x=3,y=2,size=0.1,color=black,position=left,fontscale=1.5,label=$v_3$]{33}
\Vertex[x=3,y=1,size=0.1,color=white,position=left,fontscale=1.5,label=$v_2$]{22}
\Vertex[x=3,y=0,size=0.1,color=white,position=left,fontscale=1.5,label=$v_1$]{11}
 
\Edge(55)(22) \Edge(44)(33) \Edge(11)(66)

\Vertex[x=7,y=0,size=0.1,color=black,position=right,fontscale=1.5,label=$v_6$]{666}
\Vertex[x=7,y=1,size=0.1,color=black,position=right,fontscale=1.5,label=$v_5$]{555}
\Vertex[x=7,y=2,size=0.1,color=white,position=right,fontscale=1.5,label=$v_4$]{444}

\Vertex[x=6,y=2,size=0.1,color=black,position=left,fontscale=1.5,label=$v_3$]{333}
\Vertex[x=6,y=1,size=0.1,color=white,position=left,fontscale=1.5,label=$v_2$]{222}
\Vertex[x=6,y=0,size=0.1,color=white,position=left,fontscale=1.5,label=$v_1$]{111}
 
\Edge[bend=45](333)(222) \Edge[bend=-45](444)(555) \Edge(111)(666) 
\end{tikzpicture}
\hspace*{-2cm}\caption{Kauffman diagrams $K_1,$ $K_1,$ and $K_2,$ respectively, as refered to in Example~\ref{ex:1}.}\label{eq:ex1}
\end{center}
\end{figure}

\subsection{Immanants for general product of two minors}\label{imm&nonnegativity-2}
Let $M,M'$ be multisets of the same size whose elements are $m_1\leq \dots \leq m_k$ and $m'_1\leq\dots\leq m'_k,$ respectively. The corresponding generalized minor is defined as the determinant of the submatrix $X_{M,M'}:=(x_{m_i,m_j'})_{i,j=1}^{k}$ of $X$ with possibly repeated rows and columns.
Clearly, any generalized minor of a totally nonnegative matrix is nonnegative. We consider the space of all polynomials of the form 
\begin{equation}\label{eq:space}
\sum_{(P_1,Q_1,P_2,Q_2)}c_{P_1,Q_1,P_2,Q_2} \det X_{P_1,Q_1} \det X_{P_2,Q_2} 
\end{equation}
where the sum is over all quadruples $(P_1,Q_1,P_2,Q_2)$ which satisfy $|P_1|=|Q_1|,$ $|P_2|=|Q_2|,$ and the union with multiplicities $P_1\Cup P_2 = M, Q_1\Cup Q_2 = M',$ for the pair $M,M'.$ It was shown by Rhoades--Skandera~\cite{RSkanTLImmp} that an element of this space is nonnegative for all totally nonnegative matrices $X$ if and only if it is equal to a nonnegative linear combination of Temperley--Lieb immanants of the generalized matrix $X_{M,M'}.$ We restate these connections in terms of Pl\"ucker coordinates of $\overline{X}.$
%
%
We associate to each product $\Delta_{I}(\overline{X})\Delta_{J}(\overline{X})$ a subset of the basis elements of an appropriate Temperley--Lieb algebra as follows. We label the vertices of a generic element of $\K_s$ by $v_1,\ldots,v_{2s},$ where $v_1$ is the bottom left element and enumeration is directed clockwise.

\begin{defn}\label{def:prmatch} Let $1\leq m\leq n$ be integers, and suppose $X$ is a $n \times m$ totally nonnegative matrix, and $\overline{X}$ is the representative of it under the Pl\"ucker embedding \eqref{Plucker-embed}. Let $I,J$ be $m$ element subsets of $[m+n],$ and $s\defeq |I\cap [n]|+|J \cap [n]|.$ We construct a pre-matching on the vertices of a generic element of $\K_s$ by $v_1,\ldots,v_{2s}$ as follows. We initialize $j=1,$ start with vertex $v_j=v_1,$ and then apply to it the following coloring scheme, sequentially for $i=1,2,\dots,n$:
$$
\begin{cases}
      \text{color $v_{j}$ in white,~then consider $v_{j+1},$ and increase $i$ to $i+1$} & \text{if $i \in I \setminus J$}, \\
      \text{color $v_{j}$ in black,~then consider $v_{j+1},$ and increase $i$ to $i+1$} & \text{if $i \in J \setminus I$}, \\
      \text{connect $v_{j}$ and $v_{j+1}$ by an edge,~then consider $v_{j+2}$, and increase $i$ to $i+1$} & \text{if $i \in J \cap I$}, \\
      \text{do nothing,~continue considering $v_{j}$, and increase $i$ to $i+1$}  & \text{if $i \notin J \cup I$}. \\
    \end{cases}
$$
Now, continue the same process, sequentially for $i= 1+n,n+2,\dots,m+n$ (with a slight change as shown below): 
$$
    \begin{cases}
      \text{color $v_{j}$ in white,~then consider $v_{j+1},$ and increase $i$ to $i+1$} & \text{if $i \in I \setminus J$}, \\
      \text{color $v_{j}$ in black,~~then consider $v_{j+1},$ and increase $i$ to $i+1$} & \text{if $i \in J \setminus I$}, \\
      \text{do nothing,~continue considering $v_{j},$ and increase $i$ to $i+1$} & \text{if $i \in J \cap I$}, \\
      \text{connect $v_{j}$ and $v_{j+1}$ by an edge,~then consider $v_{j+2},$ and increase $i$ to $i+1$} & \text{if $i \notin J \cup I$}. \\
    \end{cases}
$$
Call a basis element $K \in \K_s$ \emph{compatible} with the pair $(I,J)$ if its Kauffman diagram contains all the edges determined above and the remaining edges connect vertices of opposite colors. Define $\Phi(I,J)$ to be the set of all elements in $\K_s$ which are compatible with $(I,J)$. Define $E(I,J)$ to be the set of mandatory edges in $K \in \Phi(I,J)$ predetermined by $I\cap J.$ We illustrate this definition with the following example.
\end{defn}

\begin{example}\label{Ex:prematch1}
Let $m=n=7,~I=[1~4~5~6~8~9~10],$ and $J=[5~7~8~9~10~11~13]$. Then $I\cap [n]=[1~4~5~6],$~$J\cap [n]=[5~7].$ Thus, $s=|[1~4~5~6]|+|[5~7]|=6.$ We consider vertices $v_1,\dots,v_{12}$ of a generic element in $\K_{12}.$ Since $1 \in I \setminus J,$ vertex $v_1$ is white. Then $2,3 \notin I \cup J,$ so we skip them. Since $4 \in I \setminus J,$ vertex $v_2$ is white. Since $5 \in I \cap J,$ vertices $v_3$ and $v_4$ are connected by edge. Since $6 \in I \setminus J,$ vertex $v_5$ is white. Since $7 \in J \setminus I,$ vertex $v_6$ is black. 
    
Now we consider elements from the second half of $[1,14].$ Since $8,~9,~10 \in J \cap I,$ we skip them. Then $11 \in J \setminus I,$ thus vertex $v_7$ is black. Since $12 \notin I\cup J,$ we connect $v_8$ and $v_9$ by an edge. Then $13 \in J \setminus I,$ so vertex $v_{10}$ is black. Since $14 \notin I\cup J,$ we connect $v_{11}$ and $v_{12}$ by an edge. The resulting pre-matching is shown in Figure~\ref{fig1}. We discuss a more involved form of this example in the next.
\end{example}

\begin{example}\label{Ex:prematch2}
Let $m=5,~n=7,~I=[1~2~5~7~11],$ and $J=[2~3~4~6~7]$. Then $I\cap [n]=[1~2~5~7],$~$J\cap [n]=[2~3~4~6~7].$ Thus, $s=|[1~2~5~7]|+|[2~3~4~6~7]|=9.$ We consider vertices $v_1,\dots,v_{18}$ of a generic element in $\K_{18}.$ Since $1 \in I \setminus J,$ vertex $v_1$ is white. Then $2 \in I \cap J,$ so vertices $v_2$ and $v_3$ are connected by an edge. Since $3,~4 \in J \setminus I,$ vertex $v_4$ and $v_5$ are black. Since $5 \in I \setminus J,$ vertex $v_6$ is white. Since $6 \in J \setminus I,$ vertex $v_7$ is black. Since $7 \in I \cap J,$ vertices $v_8$ and $v_9$ are connected by an edge.
    
Now we consider elements from $[8,12].$ Since $8,~9,~10 \notin J \cup I,$ we have three pairs of edges $v_{10}$ and $v_{11}$, $v_{12}$ and $v_{13}$, $v_{14}$ and $v_{15}$. Then $11 \in I \setminus J,$ thus vertex $v_{16}$ is white. Since $12 \notin I\cup J,$ we connect $v_{17}$ and $v_{18}$ by an edge. The resulting pre-matching is shown in Figure~\ref{fig1}.

Note that according to the embedding \eqref{TNNGrass} $$\Delta_{[1~2~5~7~11]}(\overline{X})\Delta_{[2~3~4~6~7]}(\overline{X}) = \Delta_{[1~2~5~7],[1~3~4~5]}(X)\Delta_{[2~3~4~6~7],[1~2~3~4~5]}(X).$$ Such product is a sum of Temperley--Lieb immanants of the generalized minor of the matrix $X$ $$X_{[1~2~2~3~4~5~6~7~7],[1~1~2~3~3~4~5~6~7~7]}.$$
\end{example}

\begin{figure}
\begin{center}
\begin{tikzpicture}
\Vertex[x=1,y=0,size=0.1,color=black,position=right,fontscale=1.5,shape=rectangle,label=$v_{12}$]{12}
\Vertex[x=1,y=1,size=0.1,color=black,position=right,fontscale=1.5,shape=rectangle,label=$v_{11}$]{11}
\Vertex[x=1,y=2,size=0.1,color=black,position=right,fontscale=1.5,label=$v_{10}$]{10}
\Vertex[x=1,y=3,size=0.1,color=black,position=right,fontscale=1.5,shape=rectangle,label=$v_9$]{9}
\Vertex[x=1,y=4,size=0.1,color=black,position=right,fontscale=1.5,shape=rectangle,label=$v_8$]{8}
\Vertex[x=1,y=5,size=0.1,color=black,position=right,fontscale=1.5,label=$v_7$]{7}

\Vertex[x=0,y=5,size=0.1,color=black,position=left,fontscale=1.5,label=$v_6$]{6}
\Vertex[x=0,y=4,size=0.1,color=white,position=left,fontscale=1.5,label=$v_5$]{5}
\Vertex[x=0,y=3,size=0.1,color=black,position=left,fontscale=1.5,shape=rectangle,label=$v_4$]{4}
\Vertex[x=0,y=2,size=0.1,color=black,position=left,fontscale=1.5,shape=rectangle,label=$v_3$]{3}
\Vertex[x=0,y=1,size=0.1,color=white,position=left,fontscale=1.5,label=$v_2$]{2}
\Vertex[x=0,y=0,size=0.1,color=white,position=left,fontscale=1.5,label=$v_1$]{1}
\Edge(3)(4) \Edge(9)(8) \Edge(11)(12)

\Vertex[x=5,y=0,size=0.1,color=black,position=right,fontscale=1.5,shape=rectangle,label=$v_{18}$]{2-18}
\Vertex[x=5,y=1,size=0.1,color=black,position=right,fontscale=1.5,shape=rectangle,label=$v_{17}$]{2-17}
\Vertex[x=5,y=2,size=0.1,color=white,position=right,fontscale=1.5,label=$v_{16}$]{2-16}

\Vertex[x=5,y=3,size=0.1,color=black,position=right,fontscale=1.5,shape=rectangle,label=$v_{15}$]{2-15}
\Vertex[x=5,y=4,size=0.1,color=black,position=right,fontscale=1.5,shape=rectangle,label=$v_{14}$]{2-14}
\Vertex[x=5,y=5,size=0.1,color=black,position=right,fontscale=1.5,shape=rectangle,label=$v_{13}$]{2-13}
\Vertex[x=5,y=6,size=0.1,color=black,position=right,fontscale=1.5,shape=rectangle,label=$v_{12}$]{2-12}
\Vertex[x=5,y=7,size=0.1,color=black,position=right,fontscale=1.5,shape=rectangle,label=$v_{11}$]{2-11}
\Vertex[x=5,y=8,size=0.1,color=black,position=right,fontscale=1.5,shape=rectangle,label=$v_{10}$]{2-10}

\Vertex[x=4,y=8,size=0.1,color=black,position=left,fontscale=1.5,shape=rectangle,label=$v_9$]{2-9}
\Vertex[x=4,y=7,size=0.1,color=black,position=left,fontscale=1.5,shape=rectangle,label=$v_8$]{2-8}
\Vertex[x=4,y=6,size=0.1,color=black,position=left,fontscale=1.5,label=$v_7$]{2-7}

\Vertex[x=4,y=5,size=0.1,color=white,position=left,fontscale=1.5,label=$v_6$]{2-6}
\Vertex[x=4,y=4,size=0.1,color=black,position=left,fontscale=1.5,label=$v_5$]{2-5}
\Vertex[x=4,y=3,size=0.1,color=black,position=left,fontscale=1.5,label=$v_4$]{2-4}
\Vertex[x=4,y=2,size=0.1,color=black,position=left,fontscale=1.5,shape=rectangle,label=$v_3$]{2-3}
\Vertex[x=4,y=1,size=0.1,color=black,position=left,fontscale=1.5,shape=rectangle,label=$v_2$]{2-2}
\Vertex[x=4,y=0,size=0.1,color=white,position=left,fontscale=1.5,label=$v_1$]{2-1}

\Edge(2-2)(2-3) \Edge(2-8)(2-9) \Edge(2-10)(2-11) \Edge(2-12)(2-13) \Edge(2-14)(2-15)
\Edge(2-17)(2-18)

\end{tikzpicture}\caption{Pre-matching diagrams for Examples~\ref{Ex:prematch1} and \ref{Ex:prematch2} respectively.}\label{fig1}
\end{center}
\end{figure}
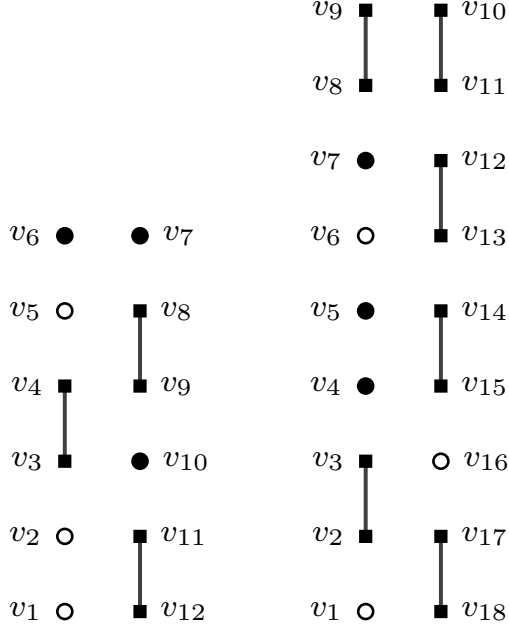


Products of the form $\Delta_{I}(\overline{X})\Delta_{J} (\overline{X})$ can be expressed as a linear combination of immanants of the generalized submatrices $X_{(P_1 \Cup P_2, Q_1 \Cup Q_2)},$ where $(P_1,Q_1)$ are the row and column indices corresponding to the Pl\"ucker coordinate $I,$ and $(P_2,Q_2)$ are the row and column indices corresponding to $J.$ With these observations and the discussion above, the following results can be stated:
\begin{thm}[Rhoades--Skandera \cite{RSkanTLImmp}]\label{tm:plprod}
Let $I,J$ and $(P_1,Q_1)$ and $(P_2,Q_2)$ be defined as above. Then we have 
\begin{equation}\label{eq:pldup}
\Delta_{I}(\overline{X})\Delta_{J} (\overline{X})=\sum_{K}\imm{K}(X_{P_1 \Cup P_2, Q_1 \Cup Q_2}),
\end{equation}    
where the sum is over elements $K \in \Phi(I,J)$.
\end{thm}

\begin{thm}[Rhoades--Skandera \cite{RSkanTLImmp}]\label{tm:plcomb}
Let $p(x_{ij})$ be a polynomial of the form 
\begin{equation}
    \sum_{(I,J)}c_{I,J} \Delta_{I}(\overline{X})\Delta_{J} (\overline{X}), 
\end{equation}
where the sum is over pairs $(I,J)$ which satisfy $I\Cup J = M$ for a fixed multiset $M$. Then $p(x)$ is totally nonnegative if and only if there exist nonnegative constants $\{d_K~| K \in \Phi(I,J) \}$ such that 
\begin{equation}\label{eq:lincom}
    p(x)=  
\sum_{K}d_{K}\imm{K}(X_{P_1 \Cup P_2, Q_1 \Cup Q_2}).
\end{equation}
\end{thm}

Based on the discussion in this subsection, as and when applicable, it would be understandable if we drop $(A)$ (or $(X)$) from $\Delta_{I}(A)$ (respectively from $\Delta_{I}(X)$), as the following discussions will be based on the indices in the inequalities. The same is applied for $\Delta_{P,Q}(A).$ Similarly, we shall use $\imm{K}$ instead of $\imm{K}(x)$ from here on.

\section{Generalized Laplace inequalities}\label{Sec:GenLapIneq}
In this section we discuss Theorem~\ref{FVthm} by restating it in terms of Pl\"ucker coordinates, and discuss its connection with long Pl\"ucker relations. Such a connection suggests a natural refinement of the inequalities, which we prove via the nonnegativity of Temperley--Lieb immanants. We consider the following:

\begin{example}\label{ex:2}
Taking $n=7$ and $d=4$ in Theorem~\ref{FVthm} yields the following set of inequalities: 
\begin{enumerate}
\item $-\Delta_{[1234],[4567]} \Delta_{[567],[123]}+\Delta_{[1234],[3567]} \Delta_{[567],[124]} \geq 0.$
\vspace*{2mm}

\item $\Delta_{[1234],[4567]} \Delta_{[567],[123]}-\Delta_{[1234],[3567]} \Delta_{[567],[124]}+\Delta_{[1234],[3467]} \Delta_{[567],[125]} \geq 0.$       
\vspace*{2mm}

\item $-\Delta_{[1234],[4567]} \Delta_{[567],[123]}+\Delta_{[1234],[3567]} \Delta_{[567],[124]}-\Delta_{[1234],[3467]} \Delta_{[567],[125]}$ \\ $+\Delta_{[1234],[3457]} \Delta_{[567],[126]} \geq 0.$
\vspace*{2mm}

\item $\Delta_{[1234],[4567]} \Delta_{[567],[123]}-\Delta_{[1234],[3567]} \Delta_{[567],[124]}+\Delta_{[1234],[3467]} \Delta_{[567],[125]} $\\$-\Delta_{[1234],[3457]} \Delta_{[567],[126]} +\Delta_{[1234],[3456]} \Delta_{[567],[127]} \geq 0.$
\end{enumerate}
\end{example}

\noindent Recall: $\Delta_{P,Q}=\Delta_{I,[m]}=\Delta_{I},$ where $I={P\cup \{m+n+1-j \,|\, j \in [m] \setminus Q\}}.$ We can restate Theorem \ref{FVthm} in terms of Pl\"ucker coordinates of $2n\times n$ matrix representatives of elements of $\Gr^{\geq 0}(n,2n)$ as follows:

\begin{thm}\label{th:5.5pl}
Let $1\leq d < n$ be integers. Suppose $I({d,k}) \defeq [1,d] \cup [n+d+2,2n] \cup \{n+d+1-k\},$ for $k \in [0,d].$ Then
\begin{equation}\label{eq:5.5pl}
    (-1)^{l}\sum^{l}_{k=0}(-1)^{k} \Delta_{I({d,k})}(\overline{A})\Delta_{I({d,k})^c}(\overline{A}) \geq 0, \quad \forall l\in [0,d],~A_{n\times n}~TNN. 
\end{equation}
\end{thm}
\noindent One can see that the expression on the left side of the inequality \eqref{eq:5.5pl} is a partial sum of the long Pl\"ucker relation. Indeed, if we extend the definition of $I_{d,l}$ for $l \in [0,n]$ then 
\begin{equation}\label{eq:pl}
 \sum^{n}_{k=0}(-1)^{n+k} \Delta_{I({d,l})}(\overline{A})\Delta_{I({d,l})^c}(\overline{A}) = 0.
\end{equation}
\noindent Therefore, it is natural to extend the inequalities in \eqref{eq:5.5pl} to all partial sums of the corresponding long Pl\"ucker relation. We achieve this by extending the range of the parameter $l,$ now $l \in [0,n].$

\begin{thm}\label{th:refined}
The notations remain as in Theorem~\ref{th:5.5pl} with $l \in [0,n].$ Then
\begin{equation}\label{eq:5.5refined}
(-1)^{l}\sum^{l}_{k=0}(-1)^{k} \Delta_{I({d,k})}(\overline{A})\Delta_{I({d,k})^c}(\overline{A}) \geq 0, \quad \forall l\in [0,n],~A_{n\times n}~TNN.
\end{equation}
\end{thm}

\noindent We prove Theorem~\ref{th:refined} after the next supplementary proposition. 

Kauffman diagrams can be viewed as a collection of $n$ noncrossing chords connecting $2n$ vertices located on the circle labeled with integers $1,2,\dots,2n.$ Therefore, cyclic shifts and reflections map Kauffman diagrams into Kauffman diagrams. This observation leads to the following:

\begin{prop}\label{pr:sym}
For an index set $\alpha,$ we define a cyclic shift
$\sigma(\alpha) = \{ i+1 \mod{2n}\, | \, i\in \alpha\}$ and a reflection $\rho(\alpha) = \{ (2n+1) -i \, | \, i \in \alpha \}.$ 
Then each of the inequalities below holds for all $A_{n\times n}~TNN$ if and only if the others do so:
\begin{align*}
\sumsb{I \subseteq [2n]\\ |I|=n} c_{I}
\Delta_{I}(\overline{A})\Delta_{I^c}(\overline{A}) \geq 0 \Leftrightarrow  \sumsb{I \subseteq [2n]\\ |I|=n} c_{I}
\Delta_{\sigma(I)}(\overline{A})\Delta_{\sigma(I)^c}(\overline{A})  \geq 0
\Leftrightarrow \sumsb{I \subseteq [2n]\\ |I|=n} c_{I}
\Delta_{\rho(I)}(\overline{A})\Delta_{\rho(I)^c}(\overline{A})  \geq 0.
\end{align*}
\end{prop}
\begin{proof}
The proof follows from Theorems \ref{SB}, \ref{t:oneprod}, and the observation that
\begin{align*}
\Delta_{I}\Delta_{I^c} =\sum_{i}b_{i}\imm{K_i} \quad\Leftrightarrow\quad \Delta_{\sigma(I)}\Delta_{\sigma(I)^c} = \sum_{i}b_{i}\imm{\sigma(K_i)}, \text{~and~}\Delta_{\rho(I)}\Delta_{\rho(I)^c} =\sum_{i}b_{i}\imm{\rho(K_i)}.
\end{align*}
Here $\sigma(K_i)$ and $\rho(K_i)$ are the images of the Kauffman diagram $K_i$ under cyclic shift $\sigma$ and reflection $\rho,$ respectively.
\end{proof}

\begin{remark}
With the discussion in Section~\ref{Sec:Pre}, it is easy to see that Proposition~\ref{pr:sym} also applies for inequalities with noncomplementary minors, as long as the inequality satisfies the homogeneity condition \eqref{eq:space}. 
\end{remark}

\begin{example}\label{ex:3}
We demonstrate Proposition~\ref{pr:sym} via an example. Recall Example~\ref{ex:1} and Figure~\ref{eq:ex1}:
\begin{align*}
\Delta_{[124]}\Delta_{[356]} - \Delta_{[123]}\Delta_{[456]}  = \big{(}\imm{K_1}+\imm{K_2}\big{)} - \imm{K_1}  = \imm{K_2}\geq 0.
\end{align*}
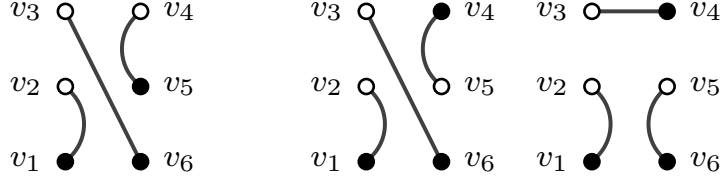
\begin{figure}
\begin{center}
\begin{tikzpicture}
\Vertex[x=0,y=0,size=0.1,color=black,position=right,fontscale=1.5,label=$v_6$]{6}
\Vertex[x=0,y=1,size=0.1,color=black,position=right,fontscale=1.5,label=$v_5$]{5}
\Vertex[x=0,y=2,size=0.1,color=white,position=right,fontscale=1.5,label=$v_4$]{4}

\Vertex[x=-1,y=2,size=0.1,color=white,position=left,fontscale=1.5,label=$v_3$]{3}
\Vertex[x=-1,y=1,size=0.1,color=white,position=left,fontscale=1.5,label=$v_2$]{2}
\Vertex[x=-1,y=0,size=0.1,color=black,position=left,fontscale=1.5,label=$v_1$]{1}

\Edge(3)(6) \Edge[bend=-45](1)(2) \Edge[bend=45](5)(4)

\Vertex[x=4,y=0,size=0.1,color=black,position=right,fontscale=1.5,label=$v_6$]{66}
\Vertex[x=4,y=1,size=0.1,color=white,position=right,fontscale=1.5,label=$v_5$]{55}
\Vertex[x=4,y=2,size=0.1,color=black,position=right,fontscale=1.5,label=$v_4$]{44}

\Vertex[x=3,y=2,size=0.1,color=white,position=left,fontscale=1.5,label=$v_3$]{33}
\Vertex[x=3,y=1,size=0.1,color=white,position=left,fontscale=1.5,label=$v_2$]{22}
\Vertex[x=3,y=0,size=0.1,color=black,position=left,fontscale=1.5,label=$v_1$]{11}
 
\Edge(33)(66) \Edge[bend=-45](11)(22) \Edge[bend=45](55)(44)

\Vertex[x=7,y=0,size=0.1,color=black,position=right,fontscale=1.5,label=$v_6$]{666}
\Vertex[x=7,y=1,size=0.1,color=white,position=right,fontscale=1.5,label=$v_5$]{555}
\Vertex[x=7,y=2,size=0.1,color=black,position=right,fontscale=1.5,label=$v_4$]{444}

\Vertex[x=6,y=2,size=0.1,color=white,position=left,fontscale=1.5,label=$v_3$]{333}
\Vertex[x=6,y=1,size=0.1,color=white,position=left,fontscale=1.5,label=$v_2$]{222}
\Vertex[x=6,y=0,size=0.1,color=black,position=left,fontscale=1.5,label=$v_1$]{111}
 
\Edge[bend=-45](111)(222) \Edge(444)(333) \Edge[bend=-45](555)(666) 
\end{tikzpicture}
\hspace*{-2cm}\caption{Kauffman diagrams $K_1',$ $K_1',$ and $ K_2',$ respectively, as referred to in Example~\ref{ex:3}.}\label{eq:ex11}
\end{center}
\end{figure}
Now compute the following using Kauffman diagrams: 
\begin{align*}
\Delta_{\sigma([124])}\Delta_{\sigma([356])} - \Delta_{\sigma([123])}\Delta_{\sigma([456])}&=\Delta_{[235]}\Delta_{[146]} - \Delta_{[234]}\Delta_{[156]}\\
&= \big{(}\imm{K'_1}+\imm{K'_2}\big{)} - \imm{K'_1}= \imm{K'_2}\geq 0.
\end{align*}
One can observe that the bi-colored noncrossing matching $K'_1$ is an image of $K_1$ under the same map $\sigma,$ and the same is true for $K'_2$ and $K_2.$ See Figure~\ref{eq:ex11} and compare with Figure~\ref{eq:ex1}.
\end{example}
Let us now apply Proposition~\ref{pr:sym} and simplify Theorem~\ref{th:refined} and prove it next:
\begin{proof}[Proof of Theorem~\ref{th:refined}] 
First we apply the cyclic shift $\sigma^{2n-d}$ and then the reflection $\rho$ on \eqref{eq:5.5refined}, to get:
\begin{equation}\label{eq:5.5refinedshifted}
 (-1)^{l}\sum^{l}_{k=0}(-1)^{k} \Delta_{[1,n-1]\cup\{n+k\}}\Delta_{[n,2n]\setminus \{n+k\}} \geq 0, \quad \mbox{for }l \in [1,n].
\end{equation}
It is sufficient to prove above to conclude this proof. Recall that for each product of minors in \eqref{eq:5.5refinedshifted} we can find all $K \in \K_n$ such that each edge of $K$ connects an element from $I=[1,n-1]\cup\{n+k\}$ to $I^c.$ For $k=0,$ we have $\Delta_{[1,n-1]\cup\{n+k\}}\Delta_{[n,2n]\setminus \{n+k\}} = \Delta_{[1,n]}\Delta_{[n+1,2n]}.$ Since $I$ consists of $n$ consecutive elements, there is exactly one $K \in \K_n$ such that each edge of $K$ connects an element from $I$ to $I^c.$ Indeed, the element $v_1$ can be connected to one of the elements on the interval $[v_{n+1}...v_{2n}],$ let us assume that $v_1$ is connected to the element $v_{2n+1-i},$ for $i\in [1,n].$ 
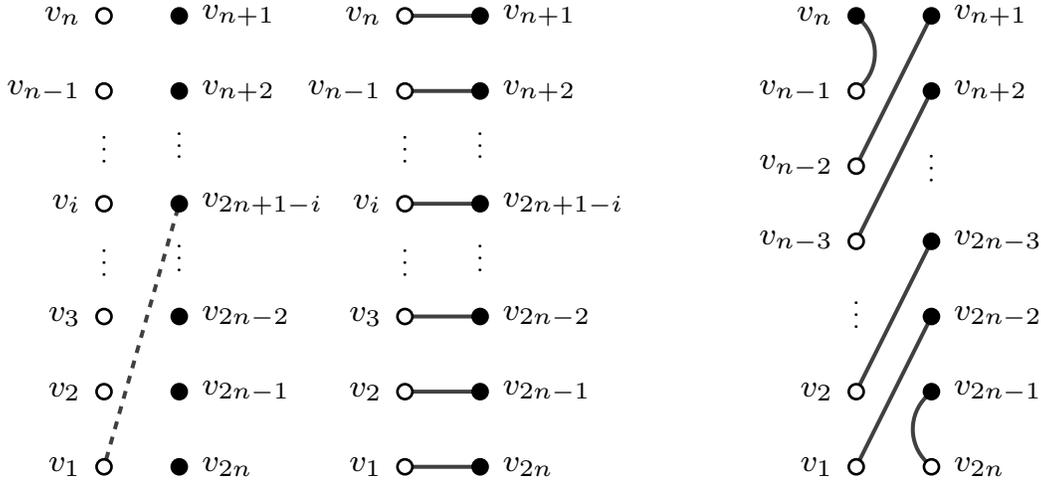
\begin{figure}
\begin{center}
\begin{tikzpicture}
\Vertex[x=1,y=0,size=0.1,color=black,position=right,fontscale=1.5,label=$v_{2n}$]{10}
\Vertex[x=1,y=1,size=0.1,color=black,position=right,fontscale=1.5,label=$v_{2n-1}$]{9}
\Vertex[x=1,y=2,size=0.1,color=black,position=right,fontscale=1.5,label=$v_{2n-2}$]{8}
\Vertex[x=1,y=3.5,size=0.1,color=black,position=right,fontscale=1.5,label=$v_{2n+1-i}$]{7}
\Vertex[x=1,y=5,size=0.1,color=black,position=right,fontscale=1.5,label=$v_{n+2}$]{60}
\Vertex[x=1,y=6,size=0.1,color=black,position=right,fontscale=1.5,label=$v_{n+1}$]{6}

\Vertex[x=0,y=6,size=0.1,color=white,position=left,fontscale=1.5,label=$v_n$]{5}
\Vertex[x=0,y=5,size=0.1,color=white,position=left,fontscale=1.5,label=$v_{n-1}$]{50}
\Vertex[x=0,y=3.5,size=0.1,color=white,position=left,fontscale=1.5,label=$v_i$]{4}

\Vertex[x=0,y=2,size=0.1,color=white,position=left,fontscale=1.5,label=$v_3$]{3}
\Vertex[x=0,y=1,size=0.1,color=white,position=left,fontscale=1.5,label=$v_2$]{2}
\Vertex[x=0,y=0,size=0.1,color=white,position=left,fontscale=1.5,label=$v_1$]{1}

\Edge[style=dashed](1)(7)

\path (3) -- (4) node [black, font=\small, midway, sloped] {$\dots$};
\path (4) -- (50) node [black, font=\small, midway, sloped] {$\dots$};
\path (7) -- (8) node [black, font=\small, midway, sloped] {$\dots$};
\path (60) -- (7) node [black, font=\small, midway, sloped] {$\dots$};

\Vertex[x=5,y=0,size=0.1,color=black,position=right,fontscale=1.5,label=$v_{2n}$]{1010}
\Vertex[x=5,y=1,size=0.1,color=black,position=right,fontscale=1.5,label=$v_{2n-1}$]{99}
\Vertex[x=5,y=2,size=0.1,color=black,position=right,fontscale=1.5,label=$v_{2n-2}$]{88}
\Vertex[x=5,y=3.5,size=0.1,color=black,position=right,fontscale=1.5,label=$v_{2n+1-i}$]{77}
\Vertex[x=5,y=5,size=0.1,color=black,position=right,fontscale=1.5,label=$v_{n+2}$]{66}
\Vertex[x=5,y=6,size=0.1,color=black,position=right,fontscale=1.5,label=$v_{n+1}$]{660}

\Vertex[x=4,y=6,size=0.1,color=white,position=left,fontscale=1.5,label=$v_n$]{550}
\Vertex[x=4,y=5,size=0.1,color=white,position=left,fontscale=1.5,label=$v_{n-1}$]{55}
\Vertex[x=4,y=3.5,size=0.1,color=white,position=left,fontscale=1.5,label=$v_i$]{44}
\Vertex[x=4,y=2,size=0.1,color=white,position=left,fontscale=1.5,label=$v_3$]{33}
\Vertex[x=4,y=1,size=0.1,color=white,position=left,fontscale=1.5,label=$v_2$]{22}
\Vertex[x=4,y=0,size=0.1,color=white,position=left,fontscale=1.5,label=$v_1$]{11}

\path (33) -- (44) node [black, font=\small, midway, sloped] {$\dots$};
\path (44) -- (55) node [black, font=\small, midway, sloped] {$\dots$};
\path (77) -- (88) node [black, font=\small, midway, sloped] {$\dots$};
\path (66) -- (77) node [black, font=\small, midway, sloped] {$\dots$};

\Edge(11)(1010) \Edge(22)(99) \Edge(33)(88) \Edge(44)(77) \Edge(55)(66) \Edge(550)(660)


\Vertex[x=11,y=0,size=0.1,color=white,position=right,fontscale=1.5,label=$v_{2n}$]{121212}
\Vertex[x=11,y=1,size=0.1,color=black,position=right,fontscale=1.5,label=$v_{2n-1}$]{111111}
\Vertex[x=11,y=2,size=0.1,color=black,position=right,fontscale=1.5,label=$v_{2n-2}$]{101010}
\Vertex[x=11,y=3,size=0.1,color=black,position=right,fontscale=1.5,label=$v_{2n-3}$]{999}
\Vertex[x=11,y=5,size=0.1,color=black,position=right,fontscale=1.5,label=$v_{n+2}$]{888}
\Vertex[x=11,y=6,size=0.1,color=black,position=right,fontscale=1.5,label=$v_{n+1}$]{777}

\Vertex[x=10,y=6,size=0.1,color=black,position=left,fontscale=1.5,label=$v_n$]{666}
\Vertex[x=10,y=5,size=0.1,color=white,position=left,fontscale=1.5,label=$v_{n-1}$]{555}
\Vertex[x=10,y=4,size=0.1,color=white,position=left,fontscale=1.5,label=$v_{n-2}$]{444}
\Vertex[x=10,y=3,size=0.1,color=white,position=left,fontscale=1.5,label=$v_{n-3}$]{333}
\Vertex[x=10,y=1,size=0.1,color=white,position=left,fontscale=1.5,label=$v_2$]{222}
\Vertex[x=10,y=0,size=0.1,color=white,position=left,fontscale=1.5,label=$v_1$]{111}

\path (333) -- (222) node [black, font=\small, midway, sloped] {$\dots$};
\path (999) -- (888) node [black, font=\small, midway, sloped] {$\dots$};

\Edge(111)(101010) \Edge(222)(999) \Edge(333)(888) \Edge(444)(777) \Edge[bend=-45](555)(666) \Edge[bend=-45](111111)(121212)

\end{tikzpicture}\caption{Unique bi-colored matching in the trivial case: see Theorem~\ref{th:refined}. On the extreme right, we have a rotated trivial diagram.}\label{d:tr0}
\end{center}
\end{figure}
Then the elements $v_{2n+2-i},\dots,v_{2n}$ (a subset of $I^c$) must be paired. Thus, the only possible option is that the set $\{v_{2n+2-i},\dots,v_{2n}\}$ is empty. This implies $v_{2n+1-i}=v_{2n}.$ By a similar argument we obtain that the remaining edges of $K$ connect $v_i$ and $v_{2n+1-i}$ for $i \in [2,n].$ We call a diagram in which $I$ and $I^c$ form two consecutive intervals a \emph{trivial diagram}, see Figure~\ref{d:tr0}. We observe that there are $n$ trivial diagrams in $\K_n$ which are obtained from rotations of the diagram in Figure~\ref{d:tr0}, i.e., obtained from the trivial diagram in Figure~\ref{d:tr0} by rotations.



Next, if $k=n$ then $\Delta_{[1,n-1]\cup\{n+k\}}\Delta_{[n,2n]\setminus \{n+k\}} = \Delta_{[1,n-1]\cup\{2n\}}\Delta_{[n,2n-1]}.$ Note that $[1,n-1]\cup \{2n\}$ and $[n,2n-1]$ form two consecutive intervals$\mod 2n.$ Thus, there is exactly one $K \in \K_n$ such that each edge of $K$ connects an element from $I$ to $I^c$, which is another trivial diagram.



Finally, if $0<k<n$ then for each product $\Delta_{[1,n-1]\cup\{n+k\}}\Delta_{[n,2n]\setminus \{n+k\}}$ there are two $K, K' \in \K_n$ such that each edge connects an element from $I$ and $I^c.$ See Figure~\ref{d:tr2}. Indeed, the element $v_{n+k} \in I$ could be connected only to an element $v_{i} \in \{ v_n~ \dots ~v_{2n}\} \setminus \{v_{n+k}\}.$ If $v_{i}$ is not one of its immediate neighbors $v_{n+k-1}$ (or $v_{n+k+1})$ then a non-empty interval $\{v_i,\ldots,v_{n+k}\}$ (or $\{v_{n+k},\ldots,v_i\}$ if $n+k<i$) is a subset of $I^c$ and separated by an edge from elements of $I.$ This is not possible for any $K \in \K_n.$ Thus, $v_{n+k}$ is connected with one of its immediate neighbors, either $v_{n+k-1}$ or $v_{n+k+1}.$ 

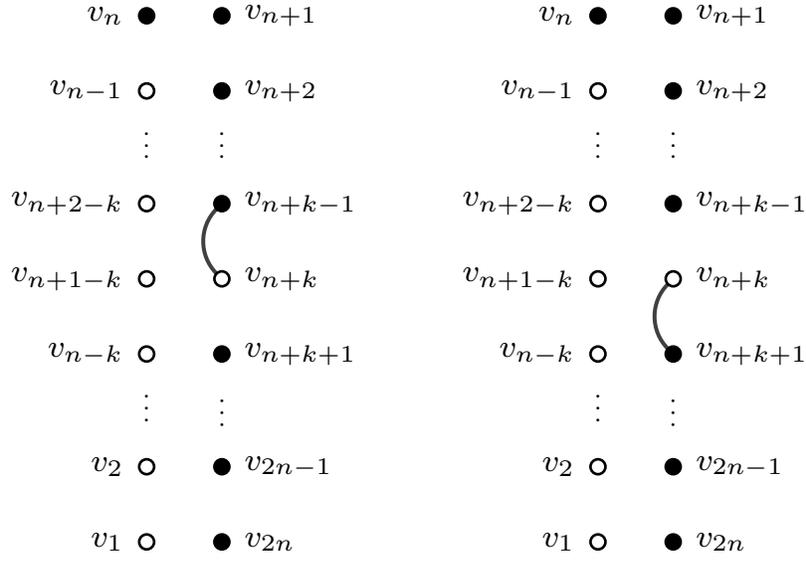
\begin{figure}
\begin{center}
\begin{tikzpicture}
\Vertex[x=1,y=0,size=0.1,color=black,position=right,fontscale=1.5,label=$v_{2n}$]{14}
\Vertex[x=1,y=1,size=0.1,color=black,position=right,fontscale=1.5,label=$v_{2n-1}$]{13}
\Vertex[x=1,y=2.5,size=0.1,color=black,position=right,fontscale=1.5,label=$v_{n+k+1}$]{12}
\Vertex[x=1,y=3.5,size=0.1,color=white,position=right,fontscale=1.5,label=$v_{n+k}$]{11}
\Vertex[x=1,y=4.5,size=0.1,color=black,position=right,fontscale=1.5,label=$v_{n+k-1}$]{10}
\Vertex[x=1,y=6,size=0.1,color=black,position=right,fontscale=1.5,label=$v_{n+2}$]{9}
\Vertex[x=1,y=7,size=0.1,color=black,position=right,fontscale=1.5,label=$v_{n+1}$]{8}

\Vertex[x=0,y=7,size=0.1,color=black,position=left,fontscale=1.5,label=$v_n$]{7}
\Vertex[x=0,y=6,size=0.1,color=white,position=left,fontscale=1.5,label=$v_{n-1}$]{6}
\Vertex[x=0,y=4.5,size=0.1,color=white,position=left,fontscale=1.5,label=$v_{n+2-k}$]{5}
\Vertex[x=0,y=3.5,size=0.1,color=white,position=left,fontscale=1.5,label=$v_{n+1-k}$]{4}
\Vertex[x=0,y=2.5,size=0.1,color=white,position=left,fontscale=1.5,label=$v_{n-k}$]{3}
\Vertex[x=0,y=1,size=0.1,color=white,position=left,fontscale=1.5,label=$v_2$]{2}
\Vertex[x=0,y=0,size=0.1,color=white,position=left,fontscale=1.5,label=$v_1$]{1}

\Edge[bend=45](11)(10)
\path (3) -- (2) node [black, font=\small, midway, sloped] {$\dots$};
\path (6) -- (5) node [black, font=\small, midway, sloped] {$\dots$};
\path (9) -- (10) node [black, font=\small, midway, sloped] {$\dots$};
\path (13) -- (12) node [black, font=\small, midway, sloped] {$\dots$};

\Vertex[x=7,y=0,size=0.1,color=black,position=right,fontscale=1.5,label=$v_{2n}$]{1414}
\Vertex[x=7,y=1,size=0.1,color=black,position=right,fontscale=1.5,label=$v_{2n-1}$]{1313}
\Vertex[x=7,y=2.5,size=0.1,color=black,position=right,fontscale=1.5,label=$v_{n+k+1}$]{1212}
\Vertex[x=7,y=3.5,size=0.1,color=white,position=right,fontscale=1.5,label=$v_{n+k}$]{1111}
\Vertex[x=7,y=4.5,size=0.1,color=black,position=right,fontscale=1.5,label=$v_{n+k-1}$]{1010}
\Vertex[x=7,y=6,size=0.1,color=black,position=right,fontscale=1.5,label=$v_{n+2}$]{99}
\Vertex[x=7,y=7,size=0.1,color=black,position=right,fontscale=1.5,label=$v_{n+1}$]{88}

\Vertex[x=6,y=7,size=0.1,color=black,position=left,fontscale=1.5,label=$v_n$]{77}
\Vertex[x=6,y=6,size=0.1,color=white,position=left,fontscale=1.5,label=$v_{n-1}$]{66}
\Vertex[x=6,y=4.5,size=0.1,color=white,position=left,fontscale=1.5,label=$v_{n+2-k}$]{55}
\Vertex[x=6,y=3.5,size=0.1,color=white,position=left,fontscale=1.5,label=$v_{n+1-k}$]{44}
\Vertex[x=6,y=2.5,size=0.1,color=white,position=left,fontscale=1.5,label=$v_{n-k}$]{33}
\Vertex[x=6,y=1,size=0.1,color=white,position=left,fontscale=1.5,label=$v_2$]{22}
\Vertex[x=6,y=0,size=0.1,color=white,position=left,fontscale=1.5,label=$v_1$]{11}

\Edge[bend=-45](1111)(1212)
\path (33) -- (22) node [black, font=\small, midway, sloped] {$\dots$};
\path (66) -- (55) node [black, font=\small, midway, sloped] {$\dots$};
\path (99) -- (1010) node [black, font=\small, midway, sloped] {$\dots$};
\path (1313) -- (1212) node [black, font=\small, midway, sloped] {$\dots$};

\end{tikzpicture}\caption{Describing intermediate stages in the proof of Theorem~\ref{th:refined}}\label{d:tr2}
\end{center}
\end{figure}


The remaining $2n-2$ vertices can be paired by edges in a unique way, since they form the trivial sub-diagram. Indeed, $I\setminus \{n+k\} = [1,n-1]$ and $I^{c} \setminus \{n+k-1\}$ (or $I^{c} \setminus \{n+k+1\}$) are consecutive intervals and the remaining $2n-2$ vertices form a trivial diagram.  


Now, any two consecutive terms in \eqref{eq:5.5refinedshifted} share exactly one matching $K \in \K_n.$ Indeed, one of the two $K, K' \in \K_n$ for each product above has an edge connecting $v_{n+k}$ and $v_{n+k+1}.$ The sub-diagrams on the remaining $2n-2$ vertices are identical. Due to Theorem~\ref{t:oneprod} the partial sum in \eqref{eq:5.5refinedshifted} can be expressed as a sum of Temperley--Lieb immanants:
\begin{align*}\label{eq:5.5refImm}
\sum^{l}_{k=0}(-1)^{1+k} & \Delta_{[1,n-1]\cup\{n+k\}}\Delta_{[n,2n]\setminus \{n+k\}} \\
&=-\imm{K_0}+\big{(}\imm{K_0}+\imm{K_1}\big{)} - \big{(}\imm{K_1}+\imm{K_2}\big{)}\\
&\hspace*{8mm} +\big{(}\imm{K_2}(x)+\imm{K_3}\big{)} + \cdots + (-1)^{l+1} \big{(}\imm{K_{l-1}}(x)+\imm{K_{l}}\big{)} = (-1)^{l+1}\imm{K_{l}},
\end{align*}
for $l\in [1,n-1].$ This implies:
$
(-1)^{1+l}\sum^{l}_{k=0}(-1)^{1+k} \Delta_{[1,n-1]\cup\{n+k\}}\Delta_{[n,2n]\setminus \{n+k\}} = \imm{K_{l}}(x) \geq 0,$ for all $l\in [1,n-1].$ For the last step, i.e., $l=n,$ recall that this inequality (in fact identity) is the long Pl\"ucker relation itself. This concludes the proof.
%
%
\end{proof}    
\begin{example}\label{ex:1.2}
Let us illustrate the proof of Theorem~\ref{th:refined} when $n=3.$ Note the telescoping nature of the partial sums:
\begin{align*}
\sum^{3}_{k=0}(-1)^{1+k} \Delta_{[1,2]\cup\{3+k\}}\Delta_{[3,6]\setminus \{3+k\}}
     &= -\imm{K_0}+\big{(}\imm{K_0}+\imm{K_1}\big{)}
     -\big{(}\imm{K_1}+\imm{K_2}\big{)}+\imm{K_2} \\
     &= 0,
\end{align*}
where the corresponding immanants, indexed by Kauffman diagrams, are shown below and in Figures~\ref{d:1},~\ref{d:2},~\ref{d:3}, and ~\ref{d:4}: 
\begin{enumerate}
\item $k=0:$ $\Delta_{[1,2,3]}\Delta_{[4,5,6]} = \imm{K_0}$

\begin{figure}
\begin{center}
\begin{tikzpicture}
\Vertex[x=1,y=0,size=0.1,color=black,position=right,fontscale=1.5,label=$v_6$]{6}
\Vertex[x=1,y=1,size=0.1,color=black,position=right,fontscale=1.5,label=$v_5$]{5}
\Vertex[x=1,y=2,size=0.1,color=black,position=right,fontscale=1.5,label=$v_4$]{4}

\Vertex[x=0,y=2,size=0.1,color=white,position=left,fontscale=1.5,label=$v_3$]{3}
\Vertex[x=0,y=1,size=0.1,color=white,position=left,fontscale=1.5,label=$v_2$]{2}
\Vertex[x=0,y=0,size=0.1,color=white,position=left,fontscale=1.5,label=$v_1$]{1}

\Edge[](1)(6) \Edge(2)(5) \Edge(3)(4)
\end{tikzpicture}
\caption{$K_0$ in Example~\ref{ex:1.2}$(1)$}\label{d:1}
\end{center}
\end{figure}
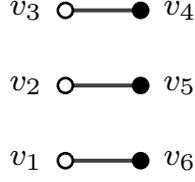
\item $k=1:$ $\Delta_{[1,2,4]}\Delta_{[3,5,6]} = \imm{K_0}+\imm{K_1}$

\begin{figure}
\begin{center}
\begin{tikzpicture}
\Vertex[x=1,y=0,size=0.1,color=black,position=right,fontscale=1.5,label=$v_6$]{6}
\Vertex[x=1,y=1,size=0.1,color=black,position=right,fontscale=1.5,label=$v_5$]{5}
\Vertex[x=1,y=2,size=0.1,color=white,position=right,fontscale=1.5,label=$v_4$]{4}

\Vertex[x=0,y=2,size=0.1,color=black,position=left,fontscale=1.5,label=$v_3$]{3}
\Vertex[x=0,y=1,size=0.1,color=white,position=left,fontscale=1.5,label=$v_2$]{2}
\Vertex[x=0,y=0,size=0.1,color=white,position=left,fontscale=1.5,label=$v_1$]{1}

\Edge[](1)(6) \Edge(2)(5) \Edge(3)(4)

\Vertex[x=5,y=0,size=0.1,color=black,position=right,fontscale=1.5,label=$v_6$]{66}
\Vertex[x=5,y=1,size=0.1,color=black,position=right,fontscale=1.5,label=$v_5$]{55}
\Vertex[x=5,y=2,size=0.1,color=white,position=right,fontscale=1.5,label=$v_4$]{44}

\Vertex[x=4,y=2,size=0.1,color=black,position=left,fontscale=1.5,label=$v_3$]{33}
\Vertex[x=4,y=1,size=0.1,color=white,position=left,fontscale=1.5,label=$v_2$]{22}
\Vertex[x=4,y=0,size=0.1,color=white,position=left,fontscale=1.5,label=$v_1$]{11}
 
\Edge[bend=45](33)(22) \Edge[bend=-45](44)(55) \Edge(11)(66)
\end{tikzpicture}
\hspace*{-2cm}\caption{$K_0$ and $K_1,$ respectively, in Example~\ref{ex:1.2}$(2).$}\label{d:2}
\end{center}
\end{figure}    

\item $k=2:$ $\Delta_{[1,2,5]}\Delta_{[3,4,6]} = \imm{K_1}+\imm{K_2}$

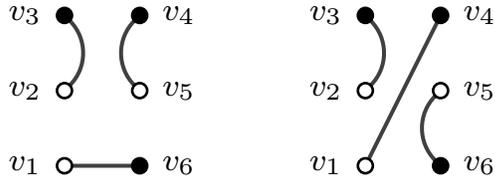
\begin{figure}
\begin{center}
\begin{tikzpicture}
\Vertex[x=1,y=0,size=0.1,color=black,position=right,fontscale=1.5,label=$v_6$]{6}
\Vertex[x=1,y=1,size=0.1,color=white,position=right,fontscale=1.5,label=$v_5$]{5}
\Vertex[x=1,y=2,size=0.1,color=black,position=right,fontscale=1.5,label=$v_4$]{4}

\Vertex[x=0,y=2,size=0.1,color=black,position=left,fontscale=1.5,label=$v_3$]{3}
\Vertex[x=0,y=1,size=0.1,color=white,position=left,fontscale=1.5,label=$v_2$]{2}
\Vertex[x=0,y=0,size=0.1,color=white,position=left,fontscale=1.5,label=$v_1$]{1}

\Edge(1)(6) \Edge[bend=-45](2)(3) \Edge[bend=45](5)(4)

\Vertex[x=5,y=0,size=0.1,color=black,position=right,fontscale=1.5,label=$v_6$]{66}
\Vertex[x=5,y=1,size=0.1,color=white,position=right,fontscale=1.5,label=$v_5$]{55}
\Vertex[x=5,y=2,size=0.1,color=black,position=right,fontscale=1.5,label=$v_4$]{44}

\Vertex[x=4,y=2,size=0.1,color=black,position=left,fontscale=1.5,label=$v_3$]{33}
\Vertex[x=4,y=1,size=0.1,color=white,position=left,fontscale=1.5,label=$v_2$]{22}
\Vertex[x=4,y=0,size=0.1,color=white,position=left,fontscale=1.5,label=$v_1$]{11}
 
\Edge(11)(44) \Edge[bend=-45](22)(33) \Edge[bend=-45](55)(66)
\end{tikzpicture}
\hspace*{-2cm}\caption{$K_1$ and $K_2,$ respectively, in Example~\ref{ex:1.2}$(3).$}\label{d:3}
\end{center}
\end{figure}

\item $k=3:$ $\Delta_{[1,2,6]}\Delta_{[3,4,5]} = \imm{K_2}$
\begin{figure}
\begin{center}
\begin{tikzpicture}
\Vertex[x=1,y=0,size=0.1,color=white,position=right,fontscale=1.5,label=$v_6$]{66}
\Vertex[x=1,y=1,size=0.1,color=black,position=right,fontscale=1.5,label=$v_5$]{55}
\Vertex[x=1,y=2,size=0.1,color=black,position=right,fontscale=1.5,label=$v_4$]{44}

\Vertex[x=0,y=2,size=0.1,color=black,position=left,fontscale=1.5,label=$v_3$]{33}
\Vertex[x=0,y=1,size=0.1,color=white,position=left,fontscale=1.5,label=$v_2$]{22}
\Vertex[x=0,y=0,size=0.1,color=white,position=left,fontscale=1.5,label=$v_1$]{11}
 
\Edge(11)(44) \Edge[bend=-45](22)(33) \Edge[bend=-45](55)(66)
\end{tikzpicture}
\hspace*{-2cm}\caption{$K_2$ in Example~\ref{ex:1.2}$(4).$}\label{d:4}
\end{center}
\end{figure}
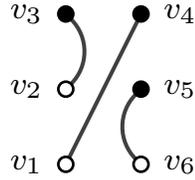
\end{enumerate}
\end{example}

\begin{example}
Let us consider an example of the inequalities in Theorem~\ref{th:refined} for $7 \times 7$ matrices $A$ and $d=4$ stated in terms of minors of $A$. Note that the first four inequalities are the same as in Example~\ref{ex:1}.
\begin{enumerate}
\item $-\Delta_{[1234],[4567]} \Delta_{[567],[123]}+\Delta_{[1234],[3567]} \Delta_{[567],[124]} \geq 0$\\
        
\item $\Delta_{[1234],[4567]} \Delta_{[567],[123]}-\Delta_{[1234],[3567]} \Delta_{[567],[124]}+\Delta_{[1234],[3467]} \Delta_{[567],[125]} \geq 0$\\
        
\item $-\Delta_{[1234],[4567]} \Delta_{[567],[123]}+\Delta_{[1234],[3567]} \Delta_{[567],[124]}-\Delta_{[1234],[3467]} \Delta_{[567],[125]}$ 

$+\Delta_{[1234],[3457]} \Delta_{[567],[126]} \geq 0$\\
        
\item $\Delta_{[1234],[4567]} \Delta_{[567],[123]}-\Delta_{[1234],[3567]} \Delta_{[567],[124]}+\Delta_{[1234],[3467]} \Delta_{[567],[125]}$ 

$-\Delta_{[1234],[3457]} \Delta_{[567],[126]} +\Delta_{[1234],[3456]} \Delta_{[567],[127]} \geq 0$\\
        
\item $-\Delta_{[1234],[4567]} \Delta_{[567],[123]}+\Delta_{[1234],[3567]} \Delta_{[567],[124]}-\Delta_{[1234],[3467]} \Delta_{[567],[125]}$ 

$+\Delta_{[1234],[3457]} \Delta_{[567],[126]} -\Delta_{[1234],[3456]} \Delta_{[567],[127]}+\Delta_{[12347],[34567]} \Delta_{[56],[12]} \geq 0$\\
        
\item $\Delta_{[1234],[4567]} \Delta_{[567],[123]}-\Delta_{[1234],[3567]} \Delta_{[567],[124]}+\Delta_{[1234],[3467]} \Delta_{[567],[125]}$ 

$-\Delta_{[1234],[3457]} \Delta_{[567],[126]} +\Delta_{[1234],[3456]} \Delta_{[567],[127]}-\Delta_{[12347],[34567]} \Delta_{[56],[12]}+$ 

$\Delta_{[12346],[34567]} \Delta_{[57],[12]}\geq 0$\\
        
\item $-\Delta_{[1234],[4567]} \Delta_{[567],[123]}+\Delta_{[1234],[3567]} \Delta_{[567],[124]}-\Delta_{[1234],[3467]} \Delta_{[567],[125]}$ 

$+\Delta_{[1234],[3457]} \Delta_{[567],[126]} -\Delta_{[1234],[3456]} \Delta_{[567],[127]}+\Delta_{[12347],[34567]} \Delta_{[56],[12]}$ 

$-\Delta_{[12346],[34567]} \Delta_{[57],[12]}+\Delta_{[12345],[34567]} \Delta_{[67],[12]}=0$
        
\end{enumerate}
\end{example}

\section{Generalization to pairs of weakly separated Pl\"ucker coordinates}\label{forward-imp}
In the theory of cluster algebras, one of the central examples is the cluster algebra of the Grassmannian $\Gr(m,m+n)$ \cite{JScott2006}. Clusters of minors are in correspondence with subsets of Pl\"ucker coordinates which can be described by set-theoretic properties. To recall, two $m$ element sets $I,J\subset [1,m+n]$ are weakly separated if $I\setminus J$ and $J \setminus I$ can be separated by a chord on the circle labeled with $1,2,\dots,m+n$ enumerated clockwise. A family of Pl\"ucker coordinates is said to be weakly separated if all pairs of Pl\"ucker coordinates in it are weakly separated \cite{LW}. It is known that maximal (by inclusion) weakly separated sets of Pl\"ucker coordinates are in bijection with the clusters consisting of Pl\"ucker coordinates in the cluster algebra of $\Gr{(m,m+n)}$ \cite{FarPost}. This makes weak separability important in the theory of cluster algebras (for instance also see \cite{OPostSpey2015}). As we discussed in the introduction, weak separability provides the exact classification of indices $I,J$ corresponding to which the Pl\"ucker-type inequalities hold. We stated this in Theorem~\ref{th:ws1}, which we prove in the next subsection.

\subsection{Inequalities for products of any two weakly separated Pl\"ucker coordinates}

We shall discuss another example before proving Theorem~\ref{th:ws1}.

\begin{example}\label{ex:22}
Suppose $m=n=6,$ and consider two weakly separated sets $I=[1~2~3~4~10~11]$ and $J=[5~6~7~8~9~11]$. Then $I \setminus J = [10~1~2~3~4]=:[i_1,i_2,i_3,i_4,i_5]$ and $J \setminus I = [5~6~7~8~9]=:[j_1,j_2,j_3,j_4,j_5].$ Therefore $\eta=5.$ Let us first write the long Pl\"ucker relation for $r=3,$ i.e. for $i_r=i_3$:
\begin{align*}
& \Delta_{[1~2~3~4~10~11]}\Delta_{[5~6~7~8~9~11]}-\Delta_{[1~3~4~5~10~11]}\Delta_{[2~6~7~8~9~11]} 
 +\Delta_{[1~3~4~6~10~11]}\Delta_{[2~5~7~8~9~11]}\\
& - \Delta_{[1~3~4~7~10~11]}\Delta_{[2~5~6~8~9~11]}
+\Delta_{[1~3~4~8~10~11]}\Delta_{[2~5~6~7~9~11]}-\Delta_{[1~3~4~9~10~11]}\Delta_{[2~5~6~7~8~11]}=0.
\end{align*}
%
Next, we express each product in the relation above as a linear combination of immanants. Each figure below presents a pre-matching with the mandatory edge $(v_{11},v_{12}),$ and the corresponding Kauffman diagrams. 
\begin{enumerate}
\item $\Delta_{[1~2~3~4~10~11]}\Delta_{[5~6~7~8~9~11]} = \imm{K_0}$
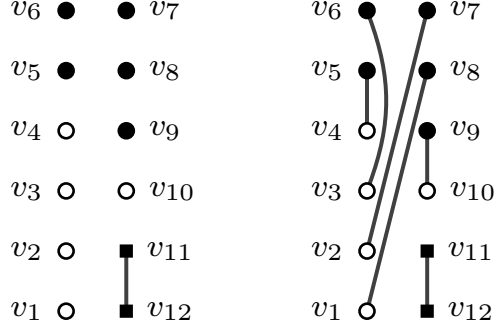
\begin{figure}
\begin{center}
\begin{tikzpicture}
\Vertex[x=0.8,y=0,size=0.1,color=black,shape=rectangle,position=right,fontscale=1.5,label=$v_{12}$]{12}
\Vertex[x=0.8,y=0.8,size=0.1,color=black,shape=rectangle,position=right,fontscale=1.5,label=$v_{11}$]{11}
\Vertex[x=0.8,y=1.6,size=0.1,color=white,position=right,fontscale=1.5,label=$v_{10}$]{10}
\Vertex[x=0.8,y=2.4,size=0.1,color=black,position=right,fontscale=1.5,label=$v_{9}$]{9}
\Vertex[x=0.8,y=3.2,size=0.1,color=black,position=right,fontscale=1.5,label=$v_{8}$]{8}
\Vertex[x=0.8,y=4.0,size=0.1,color=black,position=right,fontscale=1.5,label=$v_{7}$]{7}

\Vertex[x=0,y=4.0,size=0.1,color=black,position=left,fontscale=1.5,label=$v_6$]{6}
\Vertex[x=0,y=3.2,size=0.1,color=black,position=left,fontscale=1.5,label=$v_{5}$]{5}
\Vertex[x=0,y=2.4,size=0.1,color=white,position=left,fontscale=1.5,label=$v_{4}$]{4}
\Vertex[x=0,y=1.6,size=0.1,color=white,position=left,fontscale=1.5,label=$v_{3}$]{3}
\Vertex[x=0,y=0.8,size=0.1,color=white,position=left,fontscale=1.5,label=$v_2$]{2}
\Vertex[x=0,y=0,size=0.1,color=white,position=left,fontscale=1.5,label=$v_1$]{1}

\Edge(11)(12) 

\Vertex[x=4.8,y=0,size=0.1,color=black,shape=rectangle,position=right,fontscale=1.5,label=$v_{12}$]{1212}
\Vertex[x=4.8,y=0.8,size=0.1,color=black,shape=rectangle,position=right,fontscale=1.5,label=$v_{11}$]{1111}
\Vertex[x=4.8,y=1.6,size=0.1,color=white,position=right,fontscale=1.5,label=$v_{10}$]{1010}
\Vertex[x=4.8,y=2.4,size=0.1,color=black,position=right,fontscale=1.5,label=$v_{9}$]{99}
\Vertex[x=4.8,y=3.2,size=0.1,color=black,position=right,fontscale=1.5,label=$v_{8}$]{88}
\Vertex[x=4.8,y=4.0,size=0.1,color=black,position=right,fontscale=1.5,label=$v_{7}$]{77}

\Vertex[x=4,y=4.0,size=0.1,color=black,position=left,fontscale=1.5,label=$v_6$]{66}
\Vertex[x=4,y=3.2,size=0.1,color=black,position=left,fontscale=1.5,label=$v_{5}$]{55}
\Vertex[x=4,y=2.4,size=0.1,color=white,position=left,fontscale=1.5,label=$v_{4}$]{44}
\Vertex[x=4,y=1.6,size=0.1,color=white,position=left,fontscale=1.5,label=$v_{3}$]{33}
\Vertex[x=4,y=0.8,size=0.1,color=white,position=left,fontscale=1.5,label=$v_2$]{22}
\Vertex[x=4,y=0,size=0.1,color=white,position=left,fontscale=1.5,label=$v_1$]{11}

\Edge(1111)(1212)
\Edge(11)(88) \Edge(22)(77) \Edge[bend=-20](33)(66) \Edge(44)(55) \Edge(99)(1010) 
\end{tikzpicture}\caption{Pre-matching diagram and $K_0$ in Example~\ref{ex:22}}\label{ex:22:1}
\end{center}
\end{figure}
\item $\Delta_{[1~3~4~5~10~11]}\Delta_{[2~6~7~8~9~11]} =\imm{K_1}+\imm{K_2}$

\begin{figure}
\begin{center}
\begin{tikzpicture}
\Vertex[x=0.8,y=0,size=0.1,color=black,shape=rectangle,position=right,fontscale=1.5,label=$v_{12}$]{12}
\Vertex[x=0.8,y=0.8,size=0.1,color=black,shape=rectangle,position=right,fontscale=1.5,label=$v_{11}$]{11}
\Vertex[x=0.8,y=1.6,size=0.1,color=white,position=right,fontscale=1.5,label=$v_{10}$]{10}
\Vertex[x=0.8,y=2.4,size=0.1,color=black,position=right,fontscale=1.5,label=$v_{9}$]{9}
\Vertex[x=0.8,y=3.2,size=0.1,color=black,position=right,fontscale=1.5,label=$v_{8}$]{8}
\Vertex[x=0.8,y=4,size=0.1,color=black,position=right,fontscale=1.5,label=$v_{7}$]{7}

\Vertex[x=0,y=4,size=0.1,color=black,position=left,fontscale=1.5,label=$v_6$]{6}
\Vertex[x=0,y=3.2,size=0.1,color=white,position=left,fontscale=1.5,label=$v_{5}$]{5}
\Vertex[x=0,y=2.4,size=0.1,color=white,position=left,fontscale=1.5,label=$v_{4}$]{4}
\Vertex[x=0,y=1.6,size=0.1,color=white,position=left,fontscale=1.5,label=$v_{3}$]{3}
\Vertex[x=0,y=0.8,size=0.1,color=black,position=left,fontscale=1.5,label=$v_2$]{2}
\Vertex[x=0,y=0,size=0.1,color=white,position=left,fontscale=1.5,label=$v_1$]{1}

\Edge(11)(12) 

\Vertex[x=4.8,y=0,size=0.1,color=black,shape=rectangle,position=right,fontscale=1.5,label=$v_{12}$]{1212}
\Vertex[x=4.8,y=0.8,size=0.1,color=black,shape=rectangle,position=right,fontscale=1.5,label=$v_{11}$]{1111}
\Vertex[x=4.8,y=1.6,size=0.1,color=white,position=right,fontscale=1.5,label=$v_{10}$]{1010}
\Vertex[x=4.8,y=2.4,size=0.1,color=black,position=right,fontscale=1.5,label=$v_{9}$]{99}
\Vertex[x=4.8,y=3.2,size=0.1,color=black,position=right,fontscale=1.5,label=$v_{8}$]{88}
\Vertex[x=4.8,y=4,size=0.1,color=black,position=right,fontscale=1.5,label=$v_{7}$]{77}

\Vertex[x=4,y=4,size=0.1,color=black,position=left,fontscale=1.5,label=$v_6$]{66}
\Vertex[x=4,y=3.2,size=0.1,color=white,position=left,fontscale=1.5,label=$v_{5}$]{55}
\Vertex[x=4,y=2.4,size=0.1,color=white,position=left,fontscale=1.5,label=$v_{4}$]{44}
\Vertex[x=4,y=1.6,size=0.1,color=white,position=left,fontscale=1.5,label=$v_{3}$]{33}
\Vertex[x=4,y=0.8,size=0.1,color=black,position=left,fontscale=1.5,label=$v_2$]{22}
\Vertex[x=4,y=0,size=0.1,color=white,position=left,fontscale=1.5,label=$v_1$]{11}

\Edge(1111)(1212)
\Edge(11)(88) \Edge[bend=-20](22)(33) \Edge[bend=-45](55)(66) \Edge[bend=-5](44)(77) \Edge[bend=-20](99)(1010) 

\Vertex[x=8.8,y=0,size=0.1,color=black,shape=rectangle,position=right,fontscale=1.5,label=$v_{12}$]{121212}
\Vertex[x=8.8,y=0.8,size=0.1,color=black,shape=rectangle,position=right,fontscale=1.5,label=$v_{11}$]{111111}
\Vertex[x=8.8,y=1.6,size=0.1,color=white,position=right,fontscale=1.5,label=$v_{10}$]{101010}
\Vertex[x=8.8,y=2.4,size=0.1,color=black,position=right,fontscale=1.5,label=$v_{9}$]{999}
\Vertex[x=8.8,y=3.2,size=0.1,color=black,position=right,fontscale=1.5,label=$v_{8}$]{888}
\Vertex[x=8.8,y=4,size=0.1,color=black,position=right,fontscale=1.5,label=$v_{7}$]{777}

\Vertex[x=8,y=4,size=0.1,color=black,position=left,fontscale=1.5,label=$v_6$]{666}
\Vertex[x=8,y=3.2,size=0.1,color=white,position=left,fontscale=1.5,label=$v_{5}$]{555}
\Vertex[x=8,y=2.4,size=0.1,color=white,position=left,fontscale=1.5,label=$v_{4}$]{444}
\Vertex[x=8,y=1.6,size=0.1,color=white,position=left,fontscale=1.5,label=$v_{3}$]{333}
\Vertex[x=8,y=0.8,size=0.1,color=black,position=left,fontscale=1.5,label=$v_2$]{222}
\Vertex[x=8,y=0,size=0.1,color=white,position=left,fontscale=1.5,label=$v_1$]{111}

\Edge(111111)(121212)
\Edge[bend=-45](111)(222) \Edge(333)(888) \Edge(444)(777) \Edge[bend=-45](555)(666) \Edge[bend=-45](999)(101010) 
\end{tikzpicture}\caption{Pre-matching diagram, $K_1,$ and $K_2,$ respectively, in Example~\ref{ex:22}}\label{ex:22:2}
\end{center}
\end{figure}

\item $\Delta_{[1~3~4~6~10~11]}\Delta_{[2~5~7~8~9~11]} =\imm{K_1}+\imm{K_2}+\imm{K_3}+\imm{K_4}$

\begin{figure}
\begin{center}
\begin{tikzpicture}
\Vertex[x=0.8,y=0,size=0.1,color=black,shape=rectangle,position=right,fontscale=1.5,label=$v_{12}$]{12}
\Vertex[x=0.8,y=0.8,size=0.1,color=black,shape=rectangle,position=right,fontscale=1.5,label=$v_{11}$]{11}
\Vertex[x=0.8,y=1.6,size=0.1,color=white,position=right,fontscale=1.5,label=$v_{10}$]{10}
\Vertex[x=0.8,y=2.4,size=0.1,color=black,position=right,fontscale=1.5,label=$v_{9}$]{9}
\Vertex[x=0.8,y=3.2,size=0.1,color=black,position=right,fontscale=1.5,label=$v_{8}$]{8}
\Vertex[x=0.8,y=4,size=0.1,color=black,position=right,fontscale=1.5,label=$v_{7}$]{7}

\Vertex[x=0,y=4,size=0.1,color=white,position=left,fontscale=1.5,label=$v_6$]{6}
\Vertex[x=0,y=3.2,size=0.1,color=black,position=left,fontscale=1.5,label=$v_{5}$]{5}
\Vertex[x=0,y=2.4,size=0.1,color=white,position=left,fontscale=1.5,label=$v_{4}$]{4}
\Vertex[x=0,y=1.6,size=0.1,color=white,position=left,fontscale=1.5,label=$v_{3}$]{3}
\Vertex[x=0,y=0.8,size=0.1,color=black,position=left,fontscale=1.5,label=$v_2$]{2}
\Vertex[x=0,y=0,size=0.1,color=white,position=left,fontscale=1.5,label=$v_1$]{1}

\Edge(11)(12) 

\Vertex[x=3.8,y=0,size=0.1,color=black,shape=rectangle,position=right,fontscale=1.5,label=$v_{12}$]{1212}
\Vertex[x=3.8,y=0.8,size=0.1,color=black,shape=rectangle,position=right,fontscale=1.5,label=$v_{11}$]{1111}
\Vertex[x=3.8,y=1.6,size=0.1,color=white,position=right,fontscale=1.5,label=$v_{10}$]{1010}
\Vertex[x=3.8,y=2.4,size=0.1,color=black,position=right,fontscale=1.5,label=$v_{9}$]{99}
\Vertex[x=3.8,y=3.2,size=0.1,color=black,position=right,fontscale=1.5,label=$v_{8}$]{88}
\Vertex[x=3.8,y=4.0,size=0.1,color=black,position=right,fontscale=1.5,label=$v_{7}$]{77}

\Vertex[x=3,y=4.0,size=0.1,color=white,position=left,fontscale=1.5,label=$v_6$]{66}
\Vertex[x=3,y=3.2,size=0.1,color=black,position=left,fontscale=1.5,label=$v_{5}$]{55}
\Vertex[x=3,y=2.4,size=0.1,color=white,position=left,fontscale=1.5,label=$v_{4}$]{44}
\Vertex[x=3,y=1.6,size=0.1,color=white,position=left,fontscale=1.5,label=$v_{3}$]{33}
\Vertex[x=3,y=0.8,size=0.1,color=black,position=left,fontscale=1.5,label=$v_2$]{22}
\Vertex[x=3,y=0,size=0.1,color=white,position=left,fontscale=1.5,label=$v_1$]{11}

\Edge(1111)(1212)
\Edge(11)(88) \Edge[bend=-20](22)(33) \Edge[bend=-45](55)(66) \Edge[bend=-10](44)(77) \Edge[bend=-20](99)(1010) 

\Vertex[x=6.8,y=0,size=0.1,color=black,shape=rectangle,position=right,fontscale=1.5,label=$v_{12}$]{121212}
\Vertex[x=6.8,y=0.8,size=0.1,color=black,shape=rectangle,position=right,fontscale=1.5,label=$v_{11}$]{111111}
\Vertex[x=6.8,y=1.6,size=0.1,color=white,position=right,fontscale=1.5,label=$v_{10}$]{101010}
\Vertex[x=6.8,y=2.4,size=0.1,color=black,position=right,fontscale=1.5,label=$v_{9}$]{999}
\Vertex[x=6.8,y=3.2,size=0.1,color=black,position=right,fontscale=1.5,label=$v_{8}$]{888}
\Vertex[x=6.8,y=4,size=0.1,color=black,position=right,fontscale=1.5,label=$v_{7}$]{777}

\Vertex[x=6,y=4,size=0.1,color=white,position=left,fontscale=1.5,label=$v_6$]{666}
\Vertex[x=6,y=3.2,size=0.1,color=black,position=left,fontscale=1.5,label=$v_{5}$]{555}
\Vertex[x=6,y=2.4,size=0.1,color=white,position=left,fontscale=1.5,label=$v_{4}$]{444}
\Vertex[x=6,y=1.6,size=0.1,color=white,position=left,fontscale=1.5,label=$v_{3}$]{333}
\Vertex[x=6,y=0.8,size=0.1,color=black,position=left,fontscale=1.5,label=$v_2$]{222}
\Vertex[x=6,y=0,size=0.1,color=white,position=left,fontscale=1.5,label=$v_1$]{111}

\Edge(111111)(121212)
\Edge[bend=-45](111)(222) \Edge(333)(888) \Edge(444)(777) \Edge[bend=-45](555)(666) \Edge[bend=-45](999)(101010) 

\Vertex[x=9.8,y=0,size=0.1,color=black,shape=rectangle,position=right,fontscale=1.5,label=$v_{12}$]{12121212}
\Vertex[x=9.8,y=0.8,size=0.1,color=black,shape=rectangle,position=right,fontscale=1.5,label=$v_{11}$]{11111111}
\Vertex[x=9.8,y=1.6,size=0.1,color=white,position=right,fontscale=1.5,label=$v_{10}$]{10101010}
\Vertex[x=9.8,y=2.4,size=0.1,color=black,position=right,fontscale=1.5,label=$v_{9}$]{9999}
\Vertex[x=9.8,y=3.2,size=0.1,color=black,position=right,fontscale=1.5,label=$v_{8}$]{8888}
\Vertex[x=9.8,y=4,size=0.1,color=black,position=right,fontscale=1.5,label=$v_{7}$]{7777}

\Vertex[x=9,y=4,size=0.1,color=white,position=left,fontscale=1.5,label=$v_6$]{6666}
\Vertex[x=9,y=3.2,size=0.1,color=black,position=left,fontscale=1.5,label=$v_{5}$]{5555}
\Vertex[x=9,y=2.4,size=0.1,color=white,position=left,fontscale=1.5,label=$v_{4}$]{4444}
\Vertex[x=9,y=1.6,size=0.1,color=white,position=left,fontscale=1.5,label=$v_{3}$]{3333}
\Vertex[x=9,y=0.8,size=0.1,color=black,position=left,fontscale=1.5,label=$v_2$]{2222}
\Vertex[x=9,y=0,size=0.1,color=white,position=left,fontscale=1.5,label=$v_1$]{1111}

\Edge(11111111)(12121212)
\Edge[bend=-45](1111)(2222) \Edge(3333)(8888) \Edge[bend=-45](4444)(5555) \Edge(6666)(7777) \Edge[bend=-45](9999)(10101010) 

\Vertex[x=12.8,y=0,size=0.1,color=black,shape=rectangle,position=right,fontscale=1.5,label=$v_{12}$]{1212121212}
\Vertex[x=12.8,y=0.8,size=0.1,color=black,shape=rectangle,position=right,fontscale=1.5,label=$v_{11}$]{1111111111}
\Vertex[x=12.8,y=1.6,size=0.1,color=white,position=right,fontscale=1.5,label=$v_{10}$]{1010101010}
\Vertex[x=12.8,y=2.4,size=0.1,color=black,position=right,fontscale=1.5,label=$v_{9}$]{99999}
\Vertex[x=12.8,y=3.2,size=0.1,color=black,position=right,fontscale=1.5,label=$v_{8}$]{88888}
\Vertex[x=12.8,y=4,size=0.1,color=black,position=right,fontscale=1.5,label=$v_{7}$]{77777}

\Vertex[x=12,y=4,size=0.1,color=white,position=left,fontscale=1.5,label=$v_6$]{66666}
\Vertex[x=12,y=3.2,size=0.1,color=black,position=left,fontscale=1.5,label=$v_{5}$]{55555}
\Vertex[x=12,y=2.4,size=0.1,color=white,position=left,fontscale=1.5,label=$v_{4}$]{44444}
\Vertex[x=12,y=1.6,size=0.1,color=white,position=left,fontscale=1.5,label=$v_{3}$]{33333}
\Vertex[x=12,y=0.8,size=0.1,color=black,position=left,fontscale=1.5,label=$v_2$]{22222}
\Vertex[x=12,y=0,size=0.1,color=white,position=left,fontscale=1.5,label=$v_1$]{11111}

\Edge(1111111111)(1212121212)
\Edge(11111)(88888) \Edge[bend=20](33333)(22222) \Edge[bend=-45](44444)(55555) \Edge(77777)(66666) \Edge[bend=-25](99999)(1010101010) 
\end{tikzpicture}\caption{Pre-matching diagram, $K_1,$ $K_2,$ $K_3,$ and $K_4,$ respectively, in Example~\ref{ex:22}}\label{ex:22:3}
\end{center}
\end{figure}

\item $\Delta_{[1~3~4~7~10~11]}\Delta_{[2~5~6~8~9~11]} =\imm{K_3}+\imm{K_4}+\imm{K_5}+\imm{K_6}+\imm{K_0}$

\begin{figure}
\begin{center}
\hspace*{-0.5cm}\begin{tikzpicture}
\Vertex[x=0.8,y=0,size=0.1,color=black,shape=rectangle,position=right,fontscale=1.5,label=$v_{12}$]{P12}
\Vertex[x=0.8,y=0.8,size=0.1,color=black,shape=rectangle,position=right,fontscale=1.5,label=$v_{11}$]{P11}
\Vertex[x=0.8,y=1.6,size=0.1,color=white,position=right,fontscale=1.5,label=$v_{10}$]{P10}
\Vertex[x=0.8,y=2.4,size=0.1,color=black,position=right,fontscale=1.5,label=$v_{9}$]{P9}
\Vertex[x=0.8,y=3.2,size=0.1,color=black,position=right,fontscale=1.5,label=$v_{8}$]{P8}
\Vertex[x=0.8,y=4,size=0.1,color=white,position=right,fontscale=1.5,label=$v_{7}$]{P7}

\Vertex[x=0,y=4,size=0.1,color=black,position=left,fontscale=1.5,label=$v_6$]{P6}
\Vertex[x=0,y=3.2,size=0.1,color=black,position=left,fontscale=1.5,label=$v_{5}$]{P5}
\Vertex[x=0,y=2.4,size=0.1,color=white,position=left,fontscale=1.5,label=$v_{4}$]{P4}
\Vertex[x=0,y=1.6,size=0.1,color=white,position=left,fontscale=1.5,label=$v_{3}$]{P3}
\Vertex[x=0,y=0.8,size=0.1,color=black,position=left,fontscale=1.5,label=$v_2$]{P2}
\Vertex[x=0,y=0,size=0.1,color=white,position=left,fontscale=1.5,label=$v_1$]{P1}

\Edge(P11)(P12)

\Vertex[x=3.8,y=0,size=0.1,color=black,shape=rectangle,position=right,fontscale=1.5,label=$v_{12}$]{K012}
\Vertex[x=3.8,y=0.8,size=0.1,color=black,shape=rectangle,position=right,fontscale=1.5,label=$v_{11}$]{K011}
\Vertex[x=3.8,y=1.6,size=0.1,color=white,position=right,fontscale=1.5,label=$v_{10}$]{K010}
\Vertex[x=3.8,y=2.4,size=0.1,color=black,position=right,fontscale=1.5,label=$v_{9}$]{K09}
\Vertex[x=3.8,y=3.2,size=0.1,color=black,position=right,fontscale=1.5,label=$v_{8}$]{K08}
\Vertex[x=3.8,y=4,size=0.1,color=white,position=right,fontscale=1.5,label=$v_{7}$]{K07}

\Vertex[x=3,y=4,size=0.1,color=black,position=left,fontscale=1.5,label=$v_6$]{K06}
\Vertex[x=3,y=3.2,size=0.1,color=black,position=left,fontscale=1.5,label=$v_{5}$]{K05}
\Vertex[x=3,y=2.4,size=0.1,color=white,position=left,fontscale=1.5,label=$v_{4}$]{K04}
\Vertex[x=3,y=1.6,size=0.1,color=white,position=left,fontscale=1.5,label=$v_{3}$]{K03}
\Vertex[x=3,y=0.8,size=0.1,color=black,position=left,fontscale=1.5,label=$v_2$]{K02}
\Vertex[x=3,y=0,size=0.1,color=white,position=left,fontscale=1.5,label=$v_1$]{K01}

\Edge(K011)(K012) 
\Edge(K01)(K08) \Edge(K02)(K07) \Edge[bend=-20](K03)(K06) \Edge(K04)(K05) \Edge(K09)(K010) 

\Vertex[x=6.8,y=0,size=0.1,color=black,shape=rectangle,position=right,fontscale=1.5,label=$v_{12}$]{K312}
\Vertex[x=6.8,y=0.8,size=0.1,color=black,shape=rectangle,position=right,fontscale=1.5,label=$v_{11}$]{K311}
\Vertex[x=6.8,y=1.6,size=0.1,color=white,position=right,fontscale=1.5,label=$v_{10}$]{K310}
\Vertex[x=6.8,y=2.4,size=0.1,color=black,position=right,fontscale=1.5,label=$v_{9}$]{K39}
\Vertex[x=6.8,y=3.2,size=0.1,color=black,position=right,fontscale=1.5,label=$v_{8}$]{K38}
\Vertex[x=6.8,y=4,size=0.1,color=white,position=right,fontscale=1.5,label=$v_{7}$]{K37}

\Vertex[x=6,y=4,size=0.1,color=black,position=left,fontscale=1.5,label=$v_6$]{K36}
\Vertex[x=6,y=3.2,size=0.1,color=black,position=left,fontscale=1.5,label=$v_{5}$]{K35}
\Vertex[x=6,y=2.4,size=0.1,color=white,position=left,fontscale=1.5,label=$v_{4}$]{K34}
\Vertex[x=6,y=1.6,size=0.1,color=white,position=left,fontscale=1.5,label=$v_{3}$]{K33}
\Vertex[x=6,y=0.8,size=0.1,color=black,position=left,fontscale=1.5,label=$v_2$]{K32}
\Vertex[x=6,y=0,size=0.1,color=white,position=left,fontscale=1.5,label=$v_1$]{K31}

\Edge(K311)(K312)
\Edge[bend=-45](K31)(K32) \Edge(K33)(K38) \Edge[bend=-45](K34)(K35) \Edge(K36)(K37) \Edge[bend=-45](K39)(K310) 

\Vertex[x=9.8,y=0,size=0.1,color=black,shape=rectangle,position=right,fontscale=1.5,label=$v_{12}$]{K412}
\Vertex[x=9.8,y=0.8,size=0.1,color=black,shape=rectangle,position=right,fontscale=1.5,label=$v_{11}$]{K411}
\Vertex[x=9.8,y=1.6,size=0.1,color=white,position=right,fontscale=1.5,label=$v_{10}$]{K410}
\Vertex[x=9.8,y=2.4,size=0.1,color=black,position=right,fontscale=1.5,label=$v_{9}$]{K49}
\Vertex[x=9.8,y=3.2,size=0.1,color=black,position=right,fontscale=1.5,label=$v_{8}$]{K48}
\Vertex[x=9.8,y=4,size=0.1,color=white,position=right,fontscale=1.5,label=$v_{7}$]{K47}

\Vertex[x=9,y=4,size=0.1,color=black,position=left,fontscale=1.5,label=$v_6$]{K46}
\Vertex[x=9,y=3.2,size=0.1,color=black,position=left,fontscale=1.5,label=$v_{5}$]{K45}
\Vertex[x=9,y=2.4,size=0.1,color=white,position=left,fontscale=1.5,label=$v_{4}$]{K44}
\Vertex[x=9,y=1.6,size=0.1,color=white,position=left,fontscale=1.5,label=$v_{3}$]{K43}
\Vertex[x=9,y=0.8,size=0.1,color=black,position=left,fontscale=1.5,label=$v_2$]{K42}
\Vertex[x=9,y=0,size=0.1,color=white,position=left,fontscale=1.5,label=$v_1$]{K41}

\Edge(K411)(K412)
\Edge(K41)(K48) \Edge[bend=20](K43)(K42) \Edge[bend=-45](K44)(K45) \Edge(K47)(K46) \Edge[bend=-25](K49)(K410)

\Vertex[x=12.8,y=0,size=0.1,color=black,shape=rectangle,position=right,fontscale=1.5,label=$v_{12}$]{K512}
\Vertex[x=12.8,y=0.8,size=0.1,color=black,shape=rectangle,position=right,fontscale=1.5,label=$v_{11}$]{K511}
\Vertex[x=12.8,y=1.6,size=0.1,color=white,position=right,fontscale=1.5,label=$v_{10}$]{K510}
\Vertex[x=12.8,y=2.4,size=0.1,color=black,position=right,fontscale=1.5,label=$v_{9}$]{K59}
\Vertex[x=12.8,y=3.2,size=0.1,color=black,position=right,fontscale=1.5,label=$v_{8}$]{K58}
\Vertex[x=12.8,y=4,size=0.1,color=white,position=right,fontscale=1.5,label=$v_{7}$]{K57}

\Vertex[x=12,y=4,size=0.1,color=black,position=left,fontscale=1.5,label=$v_6$]{K56}
\Vertex[x=12,y=3.2,size=0.1,color=black,position=left,fontscale=1.5,label=$v_{5}$]{K55}
\Vertex[x=12,y=2.4,size=0.1,color=white,position=left,fontscale=1.5,label=$v_{4}$]{K54}
\Vertex[x=12,y=1.6,size=0.1,color=white,position=left,fontscale=1.5,label=$v_{3}$]{K53}
\Vertex[x=12,y=0.8,size=0.1,color=black,position=left,fontscale=1.5,label=$v_2$]{K52}
\Vertex[x=12,y=0,size=0.1,color=white,position=left,fontscale=1.5,label=$v_1$]{K51}

\Edge(K511)(K512)
\Edge[bend=-20](K51)(K56) \Edge[bend=30](K53)(K52) \Edge[bend=-30](K54)(K55) \Edge[bend=-45](K57)(K58) \Edge[bend=-30](K59)(K510)

\Vertex[x=15.8,y=0,size=0.1,color=black,shape=rectangle,position=right,fontscale=1.5,label=$v_{12}$]{K612}
\Vertex[x=15.8,y=0.8,size=0.1,color=black,shape=rectangle,position=right,fontscale=1.5,label=$v_{11}$]{K611}
\Vertex[x=15.8,y=1.6,size=0.1,color=white,position=right,fontscale=1.5,label=$v_{10}$]{K610}
\Vertex[x=15.8,y=2.4,size=0.1,color=black,position=right,fontscale=1.5,label=$v_{9}$]{K69}
\Vertex[x=15.8,y=3.2,size=0.1,color=black,position=right,fontscale=1.5,label=$v_{8}$]{K68}
\Vertex[x=15.8,y=4,size=0.1,color=white,position=right,fontscale=1.5,label=$v_{7}$]{K67}

\Vertex[x=15,y=4,size=0.1,color=black,position=left,fontscale=1.5,label=$v_6$]{K66}
\Vertex[x=15,y=3.2,size=0.1,color=black,position=left,fontscale=1.5,label=$v_{5}$]{K65}
\Vertex[x=15,y=2.4,size=0.1,color=white,position=left,fontscale=1.5,label=$v_{4}$]{K64}
\Vertex[x=15,y=1.6,size=0.1,color=white,position=left,fontscale=1.5,label=$v_{3}$]{K63}
\Vertex[x=15,y=0.8,size=0.1,color=black,position=left,fontscale=1.5,label=$v_2$]{K62}
\Vertex[x=15,y=0,size=0.1,color=white,position=left,fontscale=1.5,label=$v_1$]{K61}

\Edge(K611)(K612)

\Edge[bend=-45](K61)(K62) \Edge[bend=-30](K63)(K66) \Edge[bend=-25](K64)(K65) \Edge[bend=-45](K67)(K68) \Edge[bend=-45](K69)(K610) 

\end{tikzpicture}\caption{Pre-matching diagram, $K_0,$ $K_3,$ $K_4,$ $K_5,$ and $K_6,$ respectively, in Example~\ref{ex:22}}\label{ex:22:4}
\end{center}
\end{figure}

\item $\Delta_{[1~3~4~8~10~11]}\Delta_{[2~5~6~7~9~11]} =\imm{K_5}+\imm{K_6}+\imm{K_7}+\imm{K_8}$

\begin{figure}
\begin{center}
\hspace*{-0.5cm}\begin{tikzpicture}
\Vertex[x=0.8,y=0,size=0.1,color=black,shape=rectangle,position=right,fontscale=1.5,label=$v_{12}$]{P12}
\Vertex[x=0.8,y=0.8,size=0.1,color=black,shape=rectangle,position=right,fontscale=1.5,label=$v_{11}$]{P11}
\Vertex[x=0.8,y=1.6,size=0.1,color=white,position=right,fontscale=1.5,label=$v_{10}$]{P10}
\Vertex[x=0.8,y=2.4,size=0.1,color=black,position=right,fontscale=1.5,label=$v_{9}$]{P9}
\Vertex[x=0.8,y=3.2,size=0.1,color=white,position=right,fontscale=1.5,label=$v_{8}$]{P8}
\Vertex[x=0.8,y=4,size=0.1,color=black,position=right,fontscale=1.5,label=$v_{7}$]{P7}

\Vertex[x=0,y=4,size=0.1,color=black,position=left,fontscale=1.5,label=$v_6$]{P6}
\Vertex[x=0,y=3.2,size=0.1,color=black,position=left,fontscale=1.5,label=$v_{5}$]{P5}
\Vertex[x=0,y=2.4,size=0.1,color=white,position=left,fontscale=1.5,label=$v_{4}$]{P4}
\Vertex[x=0,y=1.6,size=0.1,color=white,position=left,fontscale=1.5,label=$v_{3}$]{P3}
\Vertex[x=0,y=0.8,size=0.1,color=black,position=left,fontscale=1.5,label=$v_2$]{P2}
\Vertex[x=0,y=0,size=0.1,color=white,position=left,fontscale=1.5,label=$v_1$]{P1}

\Edge(P11)(P12) 

\Vertex[x=3.8,y=0,size=0.1,color=black,shape=rectangle,position=right,fontscale=1.5,label=$v_{12}$]{K512}
\Vertex[x=3.8,y=0.8,size=0.1,color=black,shape=rectangle,position=right,fontscale=1.5,label=$v_{11}$]{K511}
\Vertex[x=3.8,y=1.6,size=0.1,color=white,position=right,fontscale=1.5,label=$v_{10}$]{K510}
\Vertex[x=3.8,y=2.4,size=0.1,color=black,position=right,fontscale=1.5,label=$v_{9}$]{K59}
\Vertex[x=3.8,y=3.2,size=0.1,color=white,position=right,fontscale=1.5,label=$v_{8}$]{K58}
\Vertex[x=3.8,y=4,size=0.1,color=black,position=right,fontscale=1.5,label=$v_{7}$]{K57}

\Vertex[x=3,y=4,size=0.1,color=black,position=left,fontscale=1.5,label=$v_6$]{K56}
\Vertex[x=3,y=3.2,size=0.1,color=black,position=left,fontscale=1.5,label=$v_{5}$]{K55}
\Vertex[x=3,y=2.4,size=0.1,color=white,position=left,fontscale=1.5,label=$v_{4}$]{K54}
\Vertex[x=3,y=1.6,size=0.1,color=white,position=left,fontscale=1.5,label=$v_{3}$]{K53}
\Vertex[x=3,y=0.8,size=0.1,color=black,position=left,fontscale=1.5,label=$v_2$]{K52}
\Vertex[x=3,y=0,size=0.1,color=white,position=left,fontscale=1.5,label=$v_1$]{K51}

\Edge(K511)(K512)
\Edge[bend=-20](K51)(K56) \Edge[bend=20](K53)(K52) \Edge[bend=-25](K54)(K55) \Edge[bend=-30](K57)(K58) \Edge[bend=-30](K59)(K510)

\Vertex[x=6.8,y=0,size=0.1,color=black,shape=rectangle,position=right,fontscale=1.5,label=$v_{12}$]{K612}
\Vertex[x=6.8,y=0.8,size=0.1,color=black,shape=rectangle,position=right,fontscale=1.5,label=$v_{11}$]{K611}
\Vertex[x=6.8,y=1.6,size=0.1,color=white,position=right,fontscale=1.5,label=$v_{10}$]{K610}
\Vertex[x=6.8,y=2.4,size=0.1,color=black,position=right,fontscale=1.5,label=$v_{9}$]{K69}
\Vertex[x=6.8,y=3.2,size=0.1,color=white,position=right,fontscale=1.5,label=$v_{8}$]{K68}
\Vertex[x=6.8,y=4,size=0.1,color=black,position=right,fontscale=1.5,label=$v_{7}$]{K67}

\Vertex[x=6,y=4,size=0.1,color=black,position=left,fontscale=1.5,label=$v_6$]{K66}
\Vertex[x=6,y=3.2,size=0.1,color=black,position=left,fontscale=1.5,label=$v_{5}$]{K65}
\Vertex[x=6,y=2.4,size=0.1,color=white,position=left,fontscale=1.5,label=$v_{4}$]{K64}
\Vertex[x=6,y=1.6,size=0.1,color=white,position=left,fontscale=1.5,label=$v_{3}$]{K63}
\Vertex[x=6,y=0.8,size=0.1,color=black,position=left,fontscale=1.5,label=$v_2$]{K62}
\Vertex[x=6,y=0,size=0.1,color=white,position=left,fontscale=1.5,label=$v_1$]{K61}

\Edge(K611)(K612)
\Edge[bend=-45](K61)(K62) \Edge[bend=-30](K63)(K66) \Edge[bend=-25](K64)(K65) \Edge[bend=-45](K67)(K68) \Edge[bend=-45](K69)(K610) 

\Vertex[x=9.8,y=0,size=0.1,color=black,shape=rectangle,position=right,fontscale=1.5,label=$v_{12}$]{K712}
\Vertex[x=9.8,y=0.8,size=0.1,color=black,shape=rectangle,position=right,fontscale=1.5,label=$v_{11}$]{K711}
\Vertex[x=9.8,y=1.6,size=0.1,color=white,position=right,fontscale=1.5,label=$v_{10}$]{K710}
\Vertex[x=9.8,y=2.4,size=0.1,color=black,position=right,fontscale=1.5,label=$v_{9}$]{K79}
\Vertex[x=9.8,y=3.2,size=0.1,color=white,position=right,fontscale=1.5,label=$v_{8}$]{K78}
\Vertex[x=9.8,y=4,size=0.1,color=black,position=right,fontscale=1.5,label=$v_{7}$]{K77}

\Vertex[x=9,y=4,size=0.1,color=black,position=left,fontscale=1.5,label=$v_6$]{K76}
\Vertex[x=9,y=3.2,size=0.1,color=black,position=left,fontscale=1.5,label=$v_{5}$]{K75}
\Vertex[x=9,y=2.4,size=0.1,color=white,position=left,fontscale=1.5,label=$v_{4}$]{K74}
\Vertex[x=9,y=1.6,size=0.1,color=white,position=left,fontscale=1.5,label=$v_{3}$]{K73}
\Vertex[x=9,y=0.8,size=0.1,color=black,position=left,fontscale=1.5,label=$v_2$]{K72}
\Vertex[x=9,y=0,size=0.1,color=white,position=left,fontscale=1.5,label=$v_1$]{K71}

\Edge(K711)(K712)
\Edge[bend=-20](K71)(K76) \Edge[bend=-20](K72)(K73) \Edge[bend=-25](K74)(K75) \Edge[bend=-20](K77)(K710) \Edge[bend=20](K79)(K78) 

\Vertex[x=12.8,y=0,size=0.1,color=black,shape=rectangle,position=right,fontscale=1.5,label=$v_{12}$]{K812}
\Vertex[x=12.8,y=0.8,size=0.1,color=black,shape=rectangle,position=right,fontscale=1.5,label=$v_{11}$]{K811}
\Vertex[x=12.8,y=1.6,size=0.1,color=white,position=right,fontscale=1.5,label=$v_{10}$]{K810}
\Vertex[x=12.8,y=2.4,size=0.1,color=black,position=right,fontscale=1.5,label=$v_{9}$]{K89}
\Vertex[x=12.8,y=3.2,size=0.1,color=white,position=right,fontscale=1.5,label=$v_{8}$]{K88}
\Vertex[x=12.8,y=4,size=0.1,color=black,position=right,fontscale=1.5,label=$v_{7}$]{K87}

\Vertex[x=12,y=4,size=0.1,color=black,position=left,fontscale=1.5,label=$v_6$]{K86}
\Vertex[x=12,y=3.2,size=0.1,color=black,position=left,fontscale=1.5,label=$v_{5}$]{K85}
\Vertex[x=12,y=2.4,size=0.1,color=white,position=left,fontscale=1.5,label=$v_{4}$]{K84}
\Vertex[x=12,y=1.6,size=0.1,color=white,position=left,fontscale=1.5,label=$v_{3}$]{K83}
\Vertex[x=12,y=0.8,size=0.1,color=black,position=left,fontscale=1.5,label=$v_2$]{K82}
\Vertex[x=12,y=0,size=0.1,color=white,position=left,fontscale=1.5,label=$v_1$]{K81}

\Edge(K811)(K812)
\Edge[bend=-45](K81)(K82) \Edge[bend=-25](K83)(K86) \Edge[bend=-25](K84)(K85) \Edge[bend=-25](K87)(K810) \Edge[bend=-25](K88)(K89) 

\end{tikzpicture}\caption{Pre-matching diagram, $K_5,$ $K_6,$ $K_7,$ and $K_8,$ respectively, in Example~\ref{ex:22}}\label{ex:22:5}
\end{center}
\end{figure}

\item $\Delta_{[1~3~4~9~10~11]}\Delta_{[2~5~6~7~8~11]} =\imm{K_7}+\imm{K_8}$
\begin{figure}
\begin{center}
\hspace*{-0.5cm}\begin{tikzpicture}
\Vertex[x=0.8,y=0,size=0.1,color=black,shape=rectangle,position=right,fontscale=1.5,label=$v_{12}$]{P12}
\Vertex[x=0.8,y=0.8,size=0.1,color=black,shape=rectangle,position=right,fontscale=1.5,label=$v_{11}$]{P11}
\Vertex[x=0.8,y=1.6,size=0.1,color=white,position=right,fontscale=1.5,label=$v_{10}$]{P10}
\Vertex[x=0.8,y=2.4,size=0.1,color=white,position=right,fontscale=1.5,label=$v_{9}$]{P9}
\Vertex[x=0.8,y=3.2,size=0.1,color=black,position=right,fontscale=1.5,label=$v_{8}$]{P8}
\Vertex[x=0.8,y=4,size=0.1,color=black,position=right,fontscale=1.5,label=$v_{7}$]{P7}

\Vertex[x=0,y=4,size=0.1,color=black,position=left,fontscale=1.5,label=$v_6$]{P6}
\Vertex[x=0,y=3.2,size=0.1,color=black,position=left,fontscale=1.5,label=$v_{5}$]{P5}
\Vertex[x=0,y=2.4,size=0.1,color=white,position=left,fontscale=1.5,label=$v_{4}$]{P4}
\Vertex[x=0,y=1.6,size=0.1,color=white,position=left,fontscale=1.5,label=$v_{3}$]{P3}
\Vertex[x=0,y=0.8,size=0.1,color=black,position=left,fontscale=1.5,label=$v_2$]{P2}
\Vertex[x=0,y=0,size=0.1,color=white,position=left,fontscale=1.5,label=$v_1$]{P1}

\Edge(P11)(P12) 

\Vertex[x=4.8,y=0,size=0.1,color=black,shape=rectangle,position=right,fontscale=1.5,label=$v_{12}$]{K712}
\Vertex[x=4.8,y=0.8,size=0.1,color=black,shape=rectangle,position=right,fontscale=1.5,label=$v_{11}$]{K711}
\Vertex[x=4.8,y=1.6,size=0.1,color=white,position=right,fontscale=1.5,label=$v_{10}$]{K710}
\Vertex[x=4.8,y=2.4,size=0.1,color=white,position=right,fontscale=1.5,label=$v_{9}$]{K79}
\Vertex[x=4.8,y=3.2,size=0.1,color=black,position=right,fontscale=1.5,label=$v_{8}$]{K78}
\Vertex[x=4.8,y=4,size=0.1,color=black,position=right,fontscale=1.5,label=$v_{7}$]{K77}

\Vertex[x=4,y=4,size=0.1,color=black,position=left,fontscale=1.5,label=$v_6$]{K76}
\Vertex[x=4,y=3.2,size=0.1,color=black,position=left,fontscale=1.5,label=$v_{5}$]{K75}
\Vertex[x=4,y=2.4,size=0.1,color=white,position=left,fontscale=1.5,label=$v_{4}$]{K74}
\Vertex[x=4,y=1.6,size=0.1,color=white,position=left,fontscale=1.5,label=$v_{3}$]{K73}
\Vertex[x=4,y=0.8,size=0.1,color=black,position=left,fontscale=1.5,label=$v_2$]{K72}
\Vertex[x=4,y=0,size=0.1,color=white,position=left,fontscale=1.5,label=$v_1$]{K71}

\Edge(K711)(K712)
\Edge[bend=-20](K71)(K76) \Edge[bend=-20](K72)(K73) \Edge[bend=-25](K74)(K75) \Edge[bend=-20](K77)(K710) \Edge[bend=20](K79)(K78)

\Vertex[x=8.8,y=0,size=0.1,color=black,shape=rectangle,position=right,fontscale=1.5,label=$v_{12}$]{K812}
\Vertex[x=8.8,y=0.8,size=0.1,color=black,shape=rectangle,position=right,fontscale=1.5,label=$v_{11}$]{K811}
\Vertex[x=8.8,y=1.6,size=0.1,color=white,position=right,fontscale=1.5,label=$v_{10}$]{K810}
\Vertex[x=8.8,y=2.4,size=0.1,color=white,position=right,fontscale=1.5,label=$v_{9}$]{K89}
\Vertex[x=8.8,y=3.2,size=0.1,color=black,position=right,fontscale=1.5,label=$v_{8}$]{K88}
\Vertex[x=8.8,y=4,size=0.1,color=black,position=right,fontscale=1.5,label=$v_{7}$]{K87}

\Vertex[x=8,y=4,size=0.1,color=black,position=left,fontscale=1.5,label=$v_6$]{K86}
\Vertex[x=8,y=3.2,size=0.1,color=black,position=left,fontscale=1.5,label=$v_{5}$]{K85}
\Vertex[x=8,y=2.4,size=0.1,color=white,position=left,fontscale=1.5,label=$v_{4}$]{K84}
\Vertex[x=8,y=1.6,size=0.1,color=white,position=left,fontscale=1.5,label=$v_{3}$]{K83}
\Vertex[x=8,y=0.8,size=0.1,color=black,position=left,fontscale=1.5,label=$v_2$]{K82}
\Vertex[x=8,y=0,size=0.1,color=white,position=left,fontscale=1.5,label=$v_1$]{K81}

\Edge(K811)(K812)
\Edge[bend=-45](K81)(K82) \Edge[bend=-25](K83)(K86) \Edge[bend=-25](K84)(K85) \Edge[bend=-25](K87)(K810) \Edge[bend=-25](K88)(K89) 

\end{tikzpicture}\caption{Pre-matching diagram, $K_7,$ and $K_8,$ respectively, in Example~\ref{ex:22}}\label{ex:22:6}
\end{center}
\end{figure}
\end{enumerate}

Therefore, we have the following inequalities:
\begin{enumerate}
\item $-\Delta_{[1~3~4~5~10~11]}\Delta_{[2~6~7~8~9~11]}+\Delta_{[1~3~4~6~10~11]}\Delta_{[2~5~7~8~9~11]} =\imm{K_3}+\imm{K_4}\geq 0$
\vspace*{2mm}

\item $\Delta_{[1~3~4~5~10~11]}\Delta_{[2~6~7~8~9~11]}-\Delta_{[1~3~4~6~10~11]}\Delta_{[2~5~7~8~9~11]} + \Delta_{[1~3~4~7~10~11]}\Delta_{[2~5~6~8~9~11]}$ \\
    $-\Delta_{[1~2~3~4~10~11]}\Delta_{[5~6~7~8~9~11]}~=\imm{K_5}+\imm{K_6}\geq 0$
\vspace*{2mm}
    
    \item $-\Delta_{[1~3~4~5~10~11]}\Delta_{[2~6~7~8~9~11]}+\Delta_{[1~3~4~6~10~11]}\Delta_{[2~5~7~8~9~11]} - \Delta_{[1~3~4~7~10~11]}\Delta_{[2~5~6~8~9~11]}$ 
    
    $+\Delta_{[1~2~3~4~10~11]}\Delta_{[5~6~7~8~9~11]}+\Delta_{[1~3~4~8~10~11]}\Delta_{[2~5~6~7~9~11]}~=\imm{K_7}+\imm{K_8}\geq 0$
\vspace*{2mm}
 
    \item The last inequality (in fact identity) is the long Pl\"ucker relation itself.
\end{enumerate}
\end{example}

Now we are ready to complete the proof of our main theorem. In this section, we prove one implication.

\begin{proof}[Proof of Theorem~\ref{th:ws1}: weak separability $\implies $ system of inequalities \eqref{ws11:eq3}]
Recall that $I,J$ are weakly separated $m$ element subsets in $[m+n],$ and we continue to use the notation in Definition~\ref{DIJlr}. Observe that since $I,J$ are weakly separated, the system of inequalities becomes the following, where $I^\uparrow$ is $I$ with elements arranged in the non-decreasing order:
\begin{align*}
(-1)^{l}\sum^{l}_{k=1} (-1)^{k} \Delta_{I_{k,r}^\uparrow}\Delta_{J_{k,r}^\uparrow}  & \geq 0 \quad \forall~ l < \eta-r+1, \mbox{ and} \notag\\
(-1)^{l}\Big{(}\sum^{l}_{k=1}(-1)^{k}\Delta_{I_{k,r}^\uparrow}\Delta_{J_{k,r}^\uparrow} - \Delta_{I^\uparrow}\Delta_{J^\uparrow} \Big{)}  & \geq 0 \quad \forall~ l \geq \eta-r+1.  
\end{align*}
Henceforth in this proof we assume $I^{\uparrow}=I,$ $J^{\uparrow}=J,$ $I_{k,r}^{\uparrow}=I_{k,r}$ and $J_{k,r}^{\uparrow}=J_{k,r}.$ First, we observe that the linear combinations above satisfy the homogeneity property \eqref{eq:space}. This means in terms of Pl\"ucker coordinates that for $1 \leq k \leq \eta$, $I_{k,r}\Cup J_{k,r}=I\Cup J.$ The Kauffman diagrams $K \in \Phi(I,J)$ may have some mandatory edges connecting consecutive vertices which are caused by elements in $I\cap J.$ Since $I\cap J = I_{k,r}\cap J_{k,r},$ for $1 \leq k \leq m,$ the sets of mandatory edges $E(I_{k,r},J_{k,r})$ are the same as $E(I,J)$ for $1 \leq k \leq m.$ Thus in future consideration we can disregard the set of predetermined mandatory edges. And in order to count the coefficients $\{d_K~| K \in \Phi(I,J) \},$ we can focus on the remaining $2\eta$ vertices.

We divide the rest of the proof into three parts. In the first, we consider the case when $r=\eta.$ Since $I,J$ are weakly separated, the remaining vertices (after making the pre-matching diagram) form two arcs $[i_1~i_2~\dots~i_\eta],$ $[j_1~j_2~\dots~j_\eta]$ of opposite colors. Thus, there is exactly one $K_0 \in \Phi(I,J).$ Indeed, the only option to complete $E(I,J)$ to the matching is to connect pairs $(v_{i_k},v_{j_{\eta-k+1}})$ for $1 \leq k \leq \eta.$ There are exactly two $K_{k-1},~K_{k} \in \Phi(I_{k,r},J_{k,r})$ for $1 \leq k \leq \eta-1.$ Indeed, the only option to complete $E(I_{k,r},J_{k,r})$ to the matching is to connect the pair $(v_{i_k},v_{i_{k-1}})$ or the pair $(v_{i_k},v_{i_{k+1}})$. In both cases, the remaining $2\eta-2$ vertices form two consecutive intervals of opposite colors, which could be completed to a matching in $\Phi(I_{k,r},J_{k,r})$ in only one trivial way. By a similar argument, there is exactly one $K_{\eta-1} \in \Phi(I_{\eta,\eta},J_{\eta,\eta})$. Note that $\Phi(I_{k,r},J_{k,r})$ and $\Phi(I_{k+1,r},J_{k+1,r})$ share an element $K_k$ for $1 \leq k \leq \eta-1$. Also, $\Phi(I,J)$ and $\Phi(I_{1,r},J_{1,r})$ share an element $K_0$. Due to Theorems \ref{tm:plprod} and \ref{tm:plcomb} and the argument above we have that 
\begin{equation}\label{eq:5.5refinedN}
   \Delta_{I}\Delta_{J}+\sum^{l}_{k=1}(-1)^{k} \Delta_{I_{k,r}}\Delta_{J_{k,r}} =(-1)^{l}\imm{K_l}.
\end{equation}
Since any generalized minor of a totally nonnegative $A$ is nonnegative, the proof of this case is complete.

For the next, suppose $\eta\geq 3$ and $r\in [2,\eta-1].$ Since $I,J$ are weakly separated, the remaining vertices (after drawing the pre-matching diagram) form two consecutive intervals $[i_1~i_2~\ldots~i_\eta]$, $[j_1~j_2~\ldots~j_\eta]  $ of opposite colors. Thus, there is exactly one $K_0 \in \Phi(I,J).$

There are exactly two $K_{1},~K_{2} \in \Phi(I_{1,r},J_{1,r})$. Indeed, the only option to complete $E(I_{k,r},J_{k,r})$ to the matching is to connect the pair $(v_{i_r},v_{i_{r-1}})$ or the pair $(v_{i_r},v_{i_{r+1}})$. In each of these two cases, the remaining $2\eta-2$ vertices form two consecutive intervals of the opposite colors which could be completed to a matching in $\Phi(I_{1,r},J_{1,r})$ in only one trivial way.

There are exactly four $K_{2k-3},~K_{2k-2},~K_{2k-1},~K_{2k} \in \Phi(I_{k,r},J_{k,r})$, for $1<k<\eta$ with $k \neq \eta-r+1.$ Indeed, the only option to complete $E(I_{k,r},J_{k,r})$ (for $1<k<\eta$ with $k \neq \eta-r+1$) is to connect the pair $(v_{i_r},v_{i_{r-1}})$ or the pair $(v_{i_r},v_{i_{r+1}})$ and $(v_{j_k},v_{j_{k-1}})$ or pair $(v_{j_k},v_{j_{k+1}}).$  In each of these four cases the remaining $2\eta-4$ vertices form two consecutive intervals of the opposite colors which could be completed to a matching in $\Phi(I_{k,r},J_{k,r})$ in only trivial ways.

There is one exceptional case $k=\eta-r+1.$ There are exactly five $K_{2k-3},~K_{2k-2},~K_{2k-1},~K_{2k},$
$~K_{0} \in \Phi(I_{k,r},J_{k,r}).$ Besides four matchings like in the cases $1<k<\eta,~k \neq \eta-r+1$ there is also $K_0.$ This is the only possible $k$ such that $v_{r}$ can be connected by an edge with not its immediate neighbor vertex. The last case $k=\eta$ is similar to the case $k=1$ with two diagrams $K_{2\eta-3}$ and $K_{2\eta-2}.$ Due to Theorems \ref{tm:plprod} and \ref{tm:plcomb} and the arguments above, we have that 
$$
\sum^{l}_{k=1}(-1)^{k+l} \Delta_{I_{k,r}}\Delta_{J_{k,r}}=(-1)^{l}~(\imm{K_{2l-1}}+\imm{K_{2l}}),~~~l \in [1,m-r],
$$
while for $l\in [\eta-r+1,\eta-1],$
$$
(-1)^{\eta-r+l}\Delta_{I}\Delta_{J}+\sum^{l}_{k=1}(-1)^{k+l} \Delta_{I_{k,r}}\Delta_{J_{k,r}}=(-1)^{l}~(\imm{K_{2l-1}}+\imm{K_{2l}}).$$
Finally, for $l=\eta$ the inequality (in fact identity) is the long Pl\"ucker relation itself. Since any generalized minor of a totally nonnegative $A$ is nonnegative, we complete this part of the theorem.

For the final case when $r=1,$ note that only $K_1 \in \Phi(I_{1,r},J_{1,r})$ and it corresponds to the edge $(v_{i_1},v_{i_2})$ and $(v_{j_1},v_{j_2})$ with rest of the $2\eta-4$ vertices forming the trivial matching as they form two consecutive intervals of opposite colors. Next, for $K_1,K_2\in \Phi(I_{2,r},J_{2,r}),$ where $K_2$ corresponds to the matching with $(v_{i_1},v_{i_2})$ and $(v_{j_{2}},v_{j_3})$ and, again, $K_1$ corresponds to $(v_{i_1},v_{i_2})$ and $(v_{j_1},v_{j_2}).$ In each of these cases, the rest of the $2\eta-4$ vertices form the trivial matching. Similarly, for $k\in [2,\eta-1],$ we get $K_{k},K_{k+1} \in \Phi(I_{k,r},J_{k,r}).$ $K_{k}$ corresponds to edges $(v_{i_1},v_{i_2})$ and $(v_{j_{k}},v_{j_{k-1}})$ (with the rest of the vertices forming trivial matching), and $K_{k+1}$ corresponds to the edges $(v_{i_1},v_{i_2})$ and $(v_{j_{k}},v_{j_{k+1}})$ (with the other vertices forming the trivial matching). Finally, for $k=\eta,$ we have $K_{\eta}$ and $K_{0},$ where $K_{\eta}$ corresponds to $(v_{i_1},v_{i_2})$ and $(v_{j_\eta},v_{j_{\eta-1}})$ (with the remaining vertices forming the trivial matching) and $K_{0},$ the trivial matching, on all $2\eta$ vertices. With this discussion and Theorems~\ref{tm:plprod}, \ref{tm:plcomb}, and \ref{Ineq-equivalence}, we conclude this part of the proof.
%
\end{proof}

\section{Proof of reverse implication in Theorem~\ref{th:ws1}}\label{rev-imp}

In this section, we prove that if the system of inequalities \eqref{ws11:eq3} holds, then $I,J$ are necessarily weakly separated, thereby characterizing weak separability. We begin by showing some examples of non-weakly separable sets with specific inequalities in the system~\eqref{ws11:eq3} that do not hold.


Suppose $2\leq m\leq n$ are integers, and set $\eta=2$. Let $I,J \subset [m+n]$ be $m$ element tuples that are not weakly separated, with
$I \setminus J = \{ i_1 < i_2 \}$ and
$J \setminus I = \{ j_1 < j_2 \}$. Recall from Section~\ref{Sec:Pre} that, without loss of generality, we assume that $I \cap J =\emptyset$ and take $I=(i_1,i_2)=(1,3)$ and $J=(j_1,j_2)=(2,4),$ and discuss the corresponding system of inequalities using Kauffman diagrams on $4$ vertices. The Kauffman diagrams for $\Delta_{I}\Delta_{J}$ are:
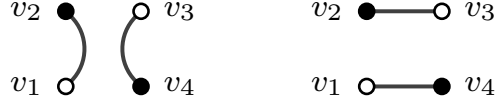
\begin{center}
\begin{tikzpicture}
\Vertex[x=1,y=1,size=0.1,color=white,position=right,fontscale=1.5,label=$v_3$]{3}
\Vertex[x=1,y=0,size=0.1,color=black,position=right,fontscale=1.5,label=$v_4$]{4}
\Vertex[x=0,y=1,size=0.1,color=black,position=left,fontscale=1.5,label=$v_2$]{2}
\Vertex[x=0,y=0,size=0.1,color=white,position=left,fontscale=1.5,label=$v_1$]{1}
\Edge[bend=-45](1)(2) \Edge[bend=-45](3)(4) 

\Vertex[x=5,y=1,size=0.1,color=white,position=right,fontscale=1.5,label=$v_3$]{33}
\Vertex[x=5,y=0,size=0.1,color=black,position=right,fontscale=1.5,label=$v_4$]{44}
\Vertex[x=4,y=1,size=0.1,color=black,position=left,fontscale=1.5,label=$v_2$]{22}
\Vertex[x=4,y=0,size=0.1,color=white,position=left,fontscale=1.5,label=$v_1$]{11}
\Edge(22)(33) \Edge(11)(44) 
\end{tikzpicture}
\end{center}
Now choose $r=\eta=2,$ i.e., $v_{i_r} = v_3$ and note that $I_{1,r}=(1,2)$ and $J_{1,r}=(3,4).$ The corresponding Kauffman diagrams are:
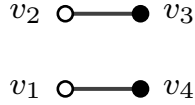
\begin{center}
\begin{tikzpicture}
\Vertex[x=4,y=1,size=0.1,color=black,position=right,fontscale=1.5,label=$v_3$]{33}
\Vertex[x=4,y=0,size=0.1,color=black,position=right,fontscale=1.5,label=$v_4$]{44}
\Vertex[x=3,y=1,size=0.1,color=white,position=left,fontscale=1.5,label=$v_2$]{22}
\Vertex[x=3,y=0,size=0.1,color=white,position=left,fontscale=1.5,label=$v_1$]{11}
\Edge(22)(33) \Edge(11)(44) 
\end{tikzpicture}
\end{center}
Looking at the diagrams above, note that the first inequality in the system, i.e. $\sgn(I_{1,r})\sgn(J_{1,r})\Pi_1 = \Pi_1 = \Delta_{I_{1,r}}\Delta_{J_{1,r}} - \Delta_I \Delta_J \geq 0$ does not hold for this non-weakly separated choice of $I,J.$ This implies that for $\eta=2,$ the system of inequalities holds for the entire TNN Grassmannian if and only if $I$ and $J$ are weakly separated.


We now examine the situation for $\eta=3.$ Again, without loss of generality, we choose a specific $I=(1,2,4)$ and $J=(3,5,6).$ The Kauffman diagrams for $\Delta_I\Delta_J$ are:
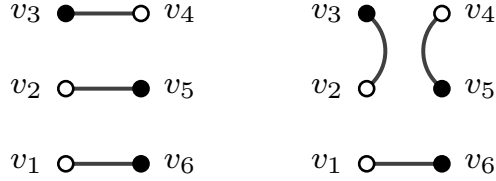
\begin{center}
\begin{tikzpicture}
\Vertex[x=1,y=0,size=0.1,color=black,position=right,fontscale=1.5,label=$v_6$]{6}
\Vertex[x=1,y=1,size=0.1,color=black,position=right,fontscale=1.5,label=$v_5$]{5}
\Vertex[x=1,y=2,size=0.1,color=white,position=right,fontscale=1.5,label=$v_4$]{4}

\Vertex[x=0,y=2,size=0.1,color=black,position=left,fontscale=1.5,label=$v_3$]{3}
\Vertex[x=0,y=1,size=0.1,color=white,position=left,fontscale=1.5,label=$v_2$]{2}
\Vertex[x=0,y=0,size=0.1,color=white,position=left,fontscale=1.5,label=$v_1$]{1}

\Edge(1)(6) \Edge(2)(5) \Edge(3)(4)

\Vertex[x=5,y=0,size=0.1,color=black,position=right,fontscale=1.5,label=$v_6$]{66}
\Vertex[x=5,y=1,size=0.1,color=black,position=right,fontscale=1.5,label=$v_5$]{55}
\Vertex[x=5,y=2,size=0.1,color=white,position=right,fontscale=1.5,label=$v_4$]{44}

\Vertex[x=4,y=2,size=0.1,color=black,position=left,fontscale=1.5,label=$v_3$]{33}
\Vertex[x=4,y=1,size=0.1,color=white,position=left,fontscale=1.5,label=$v_2$]{22}
\Vertex[x=4,y=0,size=0.1,color=white,position=left,fontscale=1.5,label=$v_1$]{11}
 
\Edge[bend=45](33)(22) \Edge[bend=-45](44)(55) \Edge(11)(66) 
\end{tikzpicture}
\end{center}
Now take $r=\eta=3,$ and draw the Kauffman diagrams for $\Delta_{I_{1,r}}\Delta_{J_{1,r}} = \Delta_{(1,2,3)}\Delta_{(4,5,6)}$:
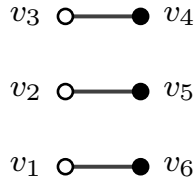
\begin{center}
\begin{tikzpicture}
\Vertex[x=1,y=0,size=0.1,color=black,position=right,fontscale=1.5,label=$v_6$]{6}
\Vertex[x=1,y=1,size=0.1,color=black,position=right,fontscale=1.5,label=$v_5$]{5}
\Vertex[x=1,y=2,size=0.1,color=black,position=right,fontscale=1.5,label=$v_4$]{4}

\Vertex[x=0,y=2,size=0.1,color=white,position=left,fontscale=1.5,label=$v_3$]{3}
\Vertex[x=0,y=1,size=0.1,color=white,position=left,fontscale=1.5,label=$v_2$]{2}
\Vertex[x=0,y=0,size=0.1,color=white,position=left,fontscale=1.5,label=$v_1$]{1}

\Edge(1)(6) \Edge(2)(5) \Edge(3)(4)
\end{tikzpicture}
\end{center}
Again, looking at the previous two Kauffman diagrams, the first inequality $\Pi_1 = \Delta_{I_{1,r}}\Delta_{J_{1,r}} - \Delta_I \Delta_J \geq 0$ does not hold.

The next essential case is $I=(1,3,5)$ and $J=(2,4,6).$ Draw the Kauffman diagrams for $\Delta_I\Delta_J$:
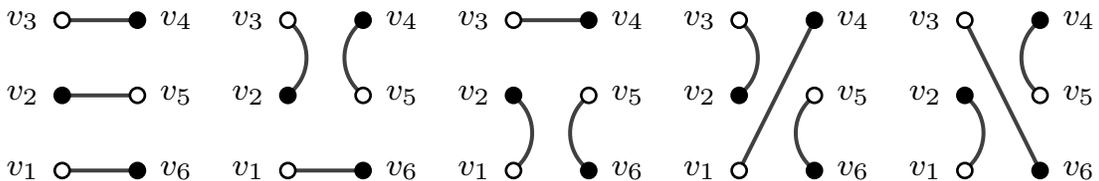
\begin{center}
\begin{tikzpicture}
\Vertex[x=1,y=0,size=0.1,color=black,position=right,fontscale=1.5,label=$v_6$]{6}
\Vertex[x=1,y=1,size=0.1,color=white,position=right,fontscale=1.5,label=$v_5$]{5}
\Vertex[x=1,y=2,size=0.1,color=black,position=right,fontscale=1.5,label=$v_4$]{4}

\Vertex[x=0,y=2,size=0.1,color=white,position=left,fontscale=1.5,label=$v_3$]{3}
\Vertex[x=0,y=1,size=0.1,color=black,position=left,fontscale=1.5,label=$v_2$]{2}
\Vertex[x=0,y=0,size=0.1,color=white,position=left,fontscale=1.5,label=$v_1$]{1}

\Edge(1)(6) \Edge(2)(5) \Edge(3)(4)

\Vertex[x=4,y=0,size=0.1,color=black,position=right,fontscale=1.5,label=$v_6$]{66}
\Vertex[x=4,y=1,size=0.1,color=white,position=right,fontscale=1.5,label=$v_5$]{55}
\Vertex[x=4,y=2,size=0.1,color=black,position=right,fontscale=1.5,label=$v_4$]{44}

\Vertex[x=3,y=2,size=0.1,color=white,position=left,fontscale=1.5,label=$v_3$]{33}
\Vertex[x=3,y=1,size=0.1,color=black,position=left,fontscale=1.5,label=$v_2$]{22}
\Vertex[x=3,y=0,size=0.1,color=white,position=left,fontscale=1.5,label=$v_1$]{11}
 
\Edge[bend=45](33)(22) \Edge[bend=-45](44)(55) \Edge(11)(66) 

\Vertex[x=7,y=0,size=0.1,color=black,position=right,fontscale=1.5,label=$v_6$]{666}
\Vertex[x=7,y=1,size=0.1,color=white,position=right,fontscale=1.5,label=$v_5$]{555}
\Vertex[x=7,y=2,size=0.1,color=black,position=right,fontscale=1.5,label=$v_4$]{444}

\Vertex[x=6,y=2,size=0.1,color=white,position=left,fontscale=1.5,label=$v_3$]{333}
\Vertex[x=6,y=1,size=0.1,color=black,position=left,fontscale=1.5,label=$v_2$]{222}
\Vertex[x=6,y=0,size=0.1,color=white,position=left,fontscale=1.5,label=$v_1$]{111}
 
\Edge(333)(444) \Edge[bend=45](666)(555) \Edge[bend=-45](111)(222) 

\Vertex[x=10,y=0,size=0.1,color=black,position=right,fontscale=1.5,label=$v_6$]{6666}
\Vertex[x=10,y=1,size=0.1,color=white,position=right,fontscale=1.5,label=$v_5$]{5555}
\Vertex[x=10,y=2,size=0.1,color=black,position=right,fontscale=1.5,label=$v_4$]{4444}

\Vertex[x=9,y=2,size=0.1,color=white,position=left,fontscale=1.5,label=$v_3$]{3333}
\Vertex[x=9,y=1,size=0.1,color=black,position=left,fontscale=1.5,label=$v_2$]{2222}
\Vertex[x=9,y=0,size=0.1,color=white,position=left,fontscale=1.5,label=$v_1$]{1111}
 
\Edge[bend=45](3333)(2222) \Edge[bend=45](6666)(5555) \Edge(1111)(4444) 

\Vertex[x=13,y=0,size=0.1,color=black,position=right,fontscale=1.5,label=$v_6$]{66666}
\Vertex[x=13,y=1,size=0.1,color=white,position=right,fontscale=1.5,label=$v_5$]{55555}
\Vertex[x=13,y=2,size=0.1,color=black,position=right,fontscale=1.5,label=$v_4$]{44444}

\Vertex[x=12,y=2,size=0.1,color=white,position=left,fontscale=1.5,label=$v_3$]{33333}
\Vertex[x=12,y=1,size=0.1,color=black,position=left,fontscale=1.5,label=$v_2$]{22222}
\Vertex[x=12,y=0,size=0.1,color=white,position=left,fontscale=1.5,label=$v_1$]{11111}
 
\Edge(33333)(66666) \Edge[bend=-45](11111)(22222) \Edge[bend=-45](44444)(55555) 
\end{tikzpicture}
\end{center}
Now take $r=\eta=3.$ We then have $I_{1,r}=(1,3,2)$ and $J_{1,r}=(5,4,6).$ The corresponding Kauffman diagrams is the trivial diagram. The first inequality in the corresponding system $$\sgn((1,3,2))\sgn((5,4,6))\Pi_1 = (-1)(-1) \big{(} \Delta_{I_{1,r}}\Delta_{J_{1,r}} - \Delta_I \Delta_J \big{)} \not\geq 0.$$ We essentially show (using Proposition~\ref{pr:sym}) that for $\eta=3$ the system of Pl\"ucker inequalities hold if and only if $I,J$ are weakly separated. 

\medskip

Now we are ready to prove the reverse implication in the proof of Theorem~\ref{th:ws1}.

\begin{proof}[Proof of Theorem~\ref{th:ws1}: system of inequalities \eqref{ws11:eq3} $\implies$ weak separability]
Suppose $1\leq \eta \leq m \leq n$ are integers, and let $I,J \subset [m+n]$ be $m$ element tuples that are not weakly separated, with $I\setminus J :=\{i_1, \dots, i_\eta\}$ and $J\setminus I :=\{j_1,\dots ,j_\eta\}$ in the notations in Definition~\ref{DIJlr}. 

\noindent\textbf{Case 1:} We checked earlier that if $\eta=2,3,$ the system of inequalities does not hold for any non-weakly separated sets. Suppose (for induction) that the system of inequalities holds for $\eta\geq 4$ and does not hold for $\eta-1$ and lower values.
Without loss of generality, assume $\eta=m.$ Thus $I$ and $J$ partition $I\cup J.$ We now show that the system does not hold for $\eta.$  Now, as $I$ and $J$ are not weakly separated, we have $p,q$ such that $i_p <_c j_q <_c i_{p+1}$ (or $j_p <_c i_q <_c j_{p+1}$) are consecutive in $I\cup J$ in the order defined in Definition~\ref{DIJlr}.

Now, consider $\Gr^{\geq 0}(m,m+n)_{i_p=j_q} \subset \Gr^{\geq 0}(m,m+n)$ as the collection of those subspaces for which the representative matrices have rows $i_p$ and $j_q$ identical. Now, since by hypothesis the inequalities hold over $\Gr^{\geq 0}(m,m+n),$ they also hold over $\Gr^{\geq 0}(m,m+n)_{i_p=j_q}.$ So, take any $r \neq p,$ i.e., take $i_r\neq i_p,$ and write the following inequalities for the given $I$ and $J$ over the Grassmannian:
\begin{align*}
\Pi_{l,r}(A) =
\begin{cases}
\Delta_{I_{l,r}}\Delta_{J_{l,r}} - \Delta_{I}\Delta_{J}, & \mbox{if }l=\eta-r+1, \mbox{ and} \\
\Delta_{I_{l,r}}\Delta_{J_{l,r}}, &\mbox{otherwise.}
\end{cases}
\end{align*}
Then $\sgn(I_{l,r})\,\sgn(J_{l,r})\sum_{k=1}^{l}\Pi_{k,r}\geq 0$ for all $l\in [1,\eta]$ and over $\Gr^{\geq 0}(m,m+n).$

\noindent Now, as the above inequalities hold over $\Gr^{\geq 0}(m,m+n)_{i_p=j_q},$ as the matrices that we have contain identical rows indexed by $i_p$ and $j_q,$ the inequalities above boil down to inequalities for $I':=I$ and $J':=(J\setminus \{j_q\})\cup\{i_p\}$ (in which $i_p$ replaces $j_q$ in $J$) with $I'\setminus J'=\{i_1',\dots,i_{\eta-1}'\}$ and $J'\setminus I'=\{j_1',\dots,j_{\eta-1}'\}.$ If $I'$ and $J'$ are not weakly separated, then we have the following system of inequalities, for $l,r'\in [\eta-1]$:
\begin{align*}
\Pi_{l,r'}(A) =
\begin{cases}
\Delta_{I_{l,r'}'}\Delta_{J_{l,r'}'} - \Delta_{I'}\Delta_{J'}, & \mbox{if }l=(\eta-1)-r'+1, \mbox{ and} \\
\Delta_{I_{l,r'}'}\Delta_{J_{l,r'}'}, &\mbox{otherwise,}
\end{cases}
\end{align*}
where $\sgn(I_{l,r'}')\,\sgn(J_{l,r'}')\sum_{k=1}^{l}\Pi_{k,r'}(A)\geq 0$ for all $l,r'\in [1,\eta-1]$ over $\Gr^{\geq 0}(m,m+n).$ This contradicts the induction hypothesis. Therefore $I$ and $J$ are necessarily weakly separated, and thus we conclude the proof. See Case 2 below if $I',J'$ are weakly separated.

\noindent\textbf{Case 2:} Without loss of generality, let $\eta=m.$ Thus $I$ and $J$ partition $I\cup J.$ Since Case 1 did not work for this pair, i.e. $I',J'$ are weakly separated, we can assume that there exist $b\in J$ and $c\in I$ such that $(I\setminus\{c\})\cup\{b\}$ and $(J\setminus\{b\})\cup\{c\}$ (in which $b$ and $c$ respectively replace each other) are contiguous intervals in $I\cup J$ in the clockwise order on the circle (defined in Definition~\ref{DIJlr}). 
Now take $r=\eta,$ i.e., $i_\eta = c.$ Then, as $j_1=b,$ $I_{1,r}=(I\setminus\{c\})\cup\{b\}$ and $J_{1,r}=(J\setminus\{b\})\cup\{c\}.$ Note that $\sgn(I_{1,r})\,\sgn(J_{1,r}) = 1,$ and thus consider the inequality 
$$
\sgn(I_{1,r})\,\sgn(J_{1,r}) \Pi_{1,r} = \Delta_{I_{1,r}}\Delta_{J_{1,r}} - \Delta_{I}\Delta_{J} \geq 0.
$$
This inequality should, in principle, hold over $\Gr^{\geq 0}(m,m+n).$ But that is not the case, which can be seen using the corresponding Kauffman diagrams (via Theorem~\ref{Ineq-equivalence}). More precisely, note that the set of Kauffman diagrams corresponding to $\Delta_{I_{1,r}}\Delta_{J_{1,r}}$ contains only the trivial diagram. However, on the other hand, $\Delta_{I}\Delta_{J}$ has more Kauffman diagrams. This proves the claim.
\end{proof}

\begin{remark}
The result in Case~2 of the above proof, or in the examples in this section above can also be proved by explicitly constructing TNN matrices using the classical Whitney's Theorem \cite{W} (or see \cite{fallat2023inequalities}). We leave this construction to the interested reader.
\end{remark}

\subsection*{Concluding remark}
One important aspect of Theorem~\ref{th:ws1} are the ``midpoints'' $\eta-r+1.$ These midpoints are unique (for each $r$) for weakly separated $I,J.$ This uniqueness plays a crucial role in proving that weak separability implies the system of Pl\"ucker-type inequalities. However, these midpoints are not unique, and in fact behave quite erratically for non-weakly separated sets. This essentially causes the converse in the main theorem to hold. The enigmatic nature of these midpoints for non-weakly separated Pl\"ucker coordinates inspires future explorations.


\section*{Acknowledgements}
We thank the American Institute of Mathematics (AIM, CalTech Campus) for their hospitality and the stimulating environment during a workshop on Theory and Applications of Total Positivity, held in July 2023. This allowed us to interact fruitfully with several other researchers which caused the beginning of this work. In particular, we thank Projesh Nath Choudhury, Shaun Fallat, Dominique Guillot, Himanshu Gupta, Chi-Kwong Li, and Shahla Nasserasr for all the encouraging and stimulating discussions during and after the AIM Workshop. We would also like to thank Michael Gekhtman for the helpful discussions, and Apoorva Khare for carefully going through previous versions of the manuscript, for suggesting improvements, and for all the encouraging discussions. We sincerely appreciate Steven Karp for writing to us and providing feedback that helped in further improving the manuscript.

P.K. Vishwkarma acknowledges the Pacific Institute for the Mathematical Sciences (PIMS) for its support. 

\vspace*{2cm}
\end{document}